\documentclass[11pt]{article} 
\usepackage[utf8]{inputenc}
\usepackage{caption,graphicx,amsmath,amsthm,amssymb, authblk,xcolor,enumitem,subfigure,float,verbatim,hyperref,pgfplots,bibentry,mathtools,soul,isomath, cancel}
\usepackage{mathrsfs}
\usepackage{indentfirst}
\usepackage{nccmath}
\graphicspath{ {images/} }
\usepackage{multicol, MnSymbol}
\usepackage{graphicx,amsmath,amssymb,amsthm,marvosym,tikz,fontenc}
\usetikzlibrary{backgrounds}
\usetikzlibrary{patterns,shapes.misc, positioning}
\usepackage{pgfplots}
\pgfplotsset{width=10cm,compat=1.9,tick scale binop=\times}
\usepackage{indentfirst}
\usepackage[a4paper,bindingoffset=0.2in,%
left=0.6in,right=0.8in,top=1.2in,bottom=1.2in,%
footskip=.25in]{geometry}
\usepackage{relsize}
\usepackage[english]{babel}
\allowdisplaybreaks
\theoremstyle{plain}
\newtheorem{theorem}{Theorem}[section]
\newtheorem{lemma}[theorem]{Lemma}
\newtheorem{proposition}[theorem]{Proposition}
\newtheorem{corollary}[theorem]{Corollary}
\newtheorem{definition}[theorem]{Definition}
\newtheorem{remark}[theorem]{Remark}
\newtheorem{example}[theorem]{Example}

\date{}
\begin{document}
\title
{\bf{On irreducibility of eccentricity matrix of graphs and construction of $\epsilon-$equienergetic graphs}}
\author {\small Anjitha Ashokan
\footnote{anjitha\_p220118ma@nitc.ac.in} and Chithra A V
\footnote{chithra@nitc.ac.in} \\ \small Department of Mathematics, National Institute of Technology, Calicut,\\
\small Kerala, India-673601}
\date{}
\maketitle
\begin{abstract} 
The eccentricity matrix $\epsilon(G)$, of a connected graph $G$ is obtained by retaining the maximum distance from each row and column of the distance matrix of $G$ and the other entries are assigned with 0.  
In this paper, we discuss the eccentricity spectrum of subdivision vertex (edge) join of regular graphs. Also, we obtain new families of graphs having irreducible or reducible eccentricity matrix. 
Furthermore, we use these results to construct infinitely many  $\epsilon-$cospectral graph pairs as well as infinitely many pairs and triplets of $\epsilon-$cospectral $\epsilon-$equienergetic graphs. Moreover, we present some new family of $\epsilon-$integral graphs. 
\end{abstract}

\section{Introduction}
Let $G= (V(G), E(G))$ be a simple, $(p, q)$ connected undirected graph with vertex set $ V(G)= \{ v_{1}, v_{2}, \ldots, v_{p} \} $, and edge set $E(G)=\{e_{1},\ldots, e_{q}\}$. If two vertices $v_{i}$ and $v_{j}$ are adjacent in $G$, we write $v_{i}\sim   v_{j}$, and the edge between them is denoted by $e_{ij}$.The vertex $v_{i}$ is incident with $e_{j}$, if $v_{i}$ is an end vertex of edge $e_{j}$. 
The degree of a vertex $v_{i}$ is denoted by $deg(v_{i})$.
A graph $G$ is $r-$regular if the degree of every vertex is $r$.
The complement $\Bar{G}$ of a graph $G$ has $V(G)$ as its vertex set and two vertices are adjacent in $\Bar{G}$ if and only if they are not adjacent in $G.$ 
The $line$ $graph$ $L(G)$ of a graph $G$, is a graph with its vertex set as the edge set of $G$ and two vertices in $L(G)$ become adjacent in $L(G)$ if the corresponding edges of $G$ share a common end vertex. The line graph of the graph $L(G)$ is represented as $L^{2}(G).$
We denote by $K_{p}$ the complete graph on $p$ vertices and by $K_{p_{1},p_{2}}$ the complete bipartite graph on $p_{1}+p_{2}$ vertices.
Throughout, we represent $J$ and $I$ as the all-one and identity matrices, respectively, in an appropriate order.
The equitable quotient matrix \cite{brouwer2011spectra} of a partitioned matrix is denoted by $F.$
The $incidence$ $matrix$ $R=R(G)$ of a graph $G$, is the $p\times q$ matrix,
\begin{equation*}
   (R(G))_{ij} =\begin{cases} 
      1& \text{if } v_{i}\text{ is incident with }e_{j},\\
      0 & \text{otherwise.}
      \end{cases}
\end{equation*}
The $adjacency$  $matrix$ $A=A(G)$ of a graph $G$,  is a square symmetric matrix,  \begin{equation*}
    (A(G))_{ij}=\begin{cases}
            1 & \text{if } v_i \sim v_j,\\
            0 &\text{ otherwise.}
            \end{cases}
\end{equation*}

The $eigenvalues $ of $A (A-eigenvalues)$ are real and ordered as $\lambda_{1}\geq \lambda_{2}\geq\ldots\geq \lambda_{p}$. 
The $A-spectrum$ of $G$, $spec (G)$, is the collection of all  $A-eigenvalues$ of $G$ along with their multiplicities.
Two non-isomorphic graphs of the same order are cospectral if they have the same $A-spectrum$.
The $energy $ $E_{A}(G)$, of a graph $ G$ is defined as, 
$E_{A}(G)=\sum_{i=1}^{p}|\lambda_{i}|$.
If the energy of two graphs with the same order is equal, they are said to be equienergetic. 
A graph $G$ is $A-$integral if all $A-$ eigenvalues of $G$ are integers.
The adjacency matrix of the complement of a graph $G$ is $A(\Bar{G})=J-I-A(G).$

The $distance$ $matrix$ $D=D(G)$ of a graph $G$ is a symmetric matrix, whose rows and columns are indexed by the set of vertices, and whose $(ij)^{th}$ entry represents the distance between the vertices $v_{i}$ and $v_{j}$. The distance between two vertices $v_{i}$ and $v_{j}$ is denoted as $d(v_{i}, v_{j}).$ 

The $eccentricity $ $e(v_{i})$ of a vertex  $v_{i}$, is defined as $e(v_{i})=max\{d(v_{i},v_{j}); v_{i}\in V(G)\}$. If $d(v_{i},v_{k})=e(v_{i}),$ then $v_{k}$ is an eccentric vertex of $v_{i}$. The minimum and maximum eccentricity of all vertices in $G$ is known as $radius(G) (r(G))$ and $diameter$ of $G(diam(G))$ respectively. 
A graph $G$ is $self-$ $centered$ if $r(G)=diam(G)$.
In \cite{randiu2013dmax}, Randi\'{c} introduced the notion of eccentricity matrix of a graph, then known as $D_{MAX}-$ matrix and later in \cite{wang2018anti}, Wang et al. renamed it as eccentricity matrix. The $eccentricity$ $matrix $, $\epsilon= \epsilon(G) $ of a graph $G$ of order $p$ is defined as follows,
\begin{equation*}
(\epsilon (G))_{ij}=\begin{cases}
                         d(v_{i}, v_{j}) &  \text{if } d(v_{i}, v_{j})=min \{\epsilon(v_{i}), \epsilon(v_{j})\},\\
                         0 & \text{ otherwise}.
                         \end{cases}
                     \\
\end{equation*}
The eccentricity and adjacency matrix of a connected graph is obtained from the distance matrix by preserving only the largest and smallest non-zero distances in each row and column, respectively. Thus, the graph's eccentricity matrix is also known as its anti-adjacency matrix.
A graph $G$ is $eccentricity$ $regular$ ($\epsilon-$regular) if the row sum of $\epsilon(G)$ is constant.
A matrix $N$ is said to be $irreducible$ if it is not permutationally similar to an upper triangular block matrix, and is $reducible$ if there exists a  permutation matrix $P$ such that\begin{equation*}
    N= P^{T}\begin{pmatrix}
        A&B\\
        0&C
    \end{pmatrix}P, 
\end{equation*}  
where $A$ and $C$ are square block matrices.


In contrast to the adjacency and distance matrices of a connected graph, the eccentricity matrix of a graph need  not be irreducible. In \cite{wang2018anti}, the authors asked the following question: For which connected graphs the eccentricity matrix is either reducible or irreducible?

The eccentricity matrix of a complete bipartite graph is reducible, whereas that of a tree is irreducible \cite{mahato2020spectra}. One of our aim is to find families of connected graphs whose eccentricity matrix is either irreducible or reducible.

Eccentricity matrix of a graph is a real symmetric matrix and its eigenvalues ($\epsilon-eigenvalues$) are real. If $\epsilon_{1}>\ldots> \epsilon_{k}$ are the distinct $\epsilon-$eigenvalues of $G$, then the eccentricity spectrum($\epsilon-$spectrum) of $G$ is defined as,\begin{equation*}
    spec_{\epsilon}(G)=\begin{pmatrix}
        \epsilon_{1}&\epsilon_{2}&\cdots&\epsilon_{k}\\
        s_{1}&s_{2}&\cdots&s_{k}
    \end{pmatrix},
\end{equation*}
where $s_{j}$ indicates the algebraic multiplicity of the $\epsilon-$eigenvalue $\epsilon_{j}$. A graph 
$G$ is eccentricity integral ($\epsilon-$integral) if all the $\epsilon-$eigenvalues are integers.
The largest eigenvalue of $\epsilon(G)$ is called the eccentricity spectral radius ($\epsilon-$spectral radius) of $G$ ($\rho_{\epsilon}(G)$).
The $\epsilon$-$Wiener$ $index$ \cite{mahato2023spectral}, $W_{\epsilon}(G)$ of a graph $G$, defined as \begin{equation*}
    W_{\epsilon}(G)=\frac{1}{2}\underset{i,j}{\sum}(\epsilon(G))_{ij}.
\end{equation*}

Two non-isomorphic graphs of the same order are said to be eccentricity cospectral ($\epsilon-$cospectral) if they have the same $\epsilon-$spectrum, otherwise, they are non $\epsilon-$cospectral. In \cite{wang2019graph}, Wang et al. introduced the concept of eccentricity energy ($\epsilon-$energy), $E_{\epsilon}(G)$ and it is defined as the sum of absolute values of $\epsilon-$ eigenvalues of $G$. Two graphs having the same order are said to be  eccentricity equienergetic ($\epsilon-$equienergetic) if they have the same $\epsilon-$energy. Obviously, $\epsilon-$cospectral graphs are 
$\epsilon-$equienergetic. 
But $\epsilon-$equienergetic graphs need not be $\epsilon-$cospectral.
Here, we focus on the construction of non isomorphic non $\epsilon-$cospectral $\epsilon-$equienergetic  graphs.

In \cite{indulal2012spectrum}, Indulal introduced the graph operations, subdivision vertex join, and subdivision edge join of two graphs and determined its adjacency spectrum. In 2019 \cite{indulal2019distance}, the authors studied about the corresponding distance spectrum. 
Motivated from this we study the $\epsilon-$spectrum of subdivision vertex join and subdivision edge join of two regular graphs.

This paper is organized as follows: In Section 2, we collect the definitions and preliminary results needed. In Sections 3 and 4, we discuss the  $\epsilon-$ spectrum of subdivision vertex(edge) join and join of graphs.  
Also, we obtain new families of graphs having either irreducible or reducible eccentricity matrix. 
Moreover, we construct families of $\epsilon-$ cospectral graphs.
In Section 5, we present infinitely many pairs and triplets of non $\epsilon-$cospectral $\epsilon-$equienergetic graphs and some families of $\epsilon-$integral graphs.

\section{Preliminaries}
In this section, we collect some basic definitions and results which will be used in the subsequent sections.

\begin{definition}\label{subdivisiongraph}\cite{cvetkovic1980spectra}
 The subdivision graph $S(G)$ of a graph $G$ is obtained from $G$ by subdividing each edge of $G$ using a new vertex. The collection of such new vertices is denoted by $I(G)$.
\end{definition}
\begin{definition}\cite{indulal2012spectrum}
    The subdivision-vertex join $G_{1}\veedot G_{2}$ of two vertex disjoint graphs $G_{1}$ and $G_{2}$ is the graph obtained from $S(G_{1})$ and $G_{2}$ by joining each vertex of $V(G_{1})$ to every vertex of $V(G_{2})$. 
\end{definition}
\begin{definition}\cite{indulal2012spectrum}
    The subdivision-edge join $G_{1} \veebar G_{2}$ of two vertex disjoint graphs $G_{1}$ and $G_{2}$
    is the graph obtained from $S(G_{1})$ and $G_{2}$ by joining each vertex of $I(G_{1})$ with every vertex of $V(G_{2})$.
\end{definition}
    
\begin{definition}\cite{sorgun2023two}
    Let $G$ be a $(p,q)$ graph. The eccentric graph of $G$ is denoted by $G^{e}$ whose vertex set is $V(G^{e})=V(G)$ and 
       $v_{i} \sim v_{j}$ in $G^{e}$ if and only if $d(v_{i},v_{j})=min\{e(v_{i}),e(v_{j})\}$.
\end{definition}
\begin{theorem}\cite{sorgun2023two}\label{irreducibilitytheorem}
    Let $G$ be any  $(p,q)$ graph. Then the matrix $\epsilon(G)$ is an irreducible if and only if $G^{e}$ is a  connected graph.
\end{theorem}
\begin{lemma}\cite{cvetkovic1980spectra}\label{thelinegraph}
    Let $G$ be an $r-$regular, $(p,q)$ graph with an adjacency matrix $A$ and an incidence matrix $R$. Let $L(G)$ be the line graph of $G$. Then $RR^{T}=A+rI$ , $R^{T}R=B+2I$, where $B$ is the adjacency matrix of $L(G)$. Also, if $J$ is an all-one  matrix of appropriate order, then $JR=2J=R^{T}J$ and $JR^{T}=rJ=RJ.$ 
\end{lemma}
\begin{lemma}\cite{cvetkovic1980spectra}\label{Beigenvalue}
    Let $G$ be an $r-$regular, $(p,q)$ graph with $spec(G)=\{r,\lambda_{2},\ldots, \lambda_{p}\}$. Then 
                   \begin{equation*}
                       spec(L(G))=\begin{pmatrix}
                           2r-2 & \lambda_{2}+r-2 & \cdots&\lambda_{p}+r-2& -2\\
                           1 & 1 &\cdots& 1 & q-p
                       \end{pmatrix}.
                   \end{equation*}
 Also, $Z$ is an eigenvector corresponding to the eigenvalue $-2$ if and only if $RZ=0$.
\end{lemma}
\begin{lemma}\cite{indulal2008equienergetic}\label{cubicnoncospectral} 
    For every $t\geq 3$, there exists a pair of non-cospectral cubic graphs on $2t$ vertices. 
\end{lemma}
\begin{lemma}\cite{indulal2008equienergetic}\label{2linegraph}
    Let $G$ be a $r-regular$ ($r\geq 3$), $(p,q)$  graph with $spec(G)=\{r, \lambda_{2},\ldots, \lambda_{p}\} $. Then
    \begin{equation*}
                       spec(L^{2}(G))=\begin{pmatrix}
                           4r-6 & \lambda_{2}+3r-6 &\cdots&\lambda_{p}+3r-6& 2r-6 & -2\\
                           1 & 1 &\cdots& 1 & \frac{1}{2}p(r-2)&\frac{1}{2}pr(r-2)
                       \end{pmatrix}.
                   \end{equation*}
\end{lemma}
\begin{lemma}\cite{cvetkovic1980spectra}
    Let $G$ be a $r-regular$, ($p, q$) graph. If the 
    $spec(G)=\{ r, \lambda_{2}, \ldots, \lambda_{p}\},$
     then the spectrum of the complement of $G$ is, 
    $spec(\Bar{G})=
        \{ p-r-1, -(\lambda_{2}+1),\ldots, -(\lambda_{p}+1)\}$.
    \end{lemma}
 \begin{theorem}\cite{mahato2023spectral}\label{spectralradiuslowerbound}
   Let $G$ be a connected, $(p,q)$ graph. Then $\rho_{\epsilon}(\epsilon(G))\geq 2\frac{W_{\epsilon}}{p}$. Moreover, equality holds if and only if $G$ is $\epsilon-$ regular.   
 \end{theorem}   


\section{Irreducibility and  eccentricity spectrum of subdivision vertex join of two graphs}

In this section, we discuss about the irreducibility, $\epsilon-$ spectrum, and eccentricity Wiener index of subdivision vertex join of two graphs.

\begin{theorem} Let $G_{i}$ be a $(p_{i}, q_{i})$ graph, $i={1,2}$.  
Then $\epsilon(G_{1} \veedot G_{2})$ is irreducible, unless $G_{1}=K_{1,1}.$
\end{theorem}
\begin{proof}

 Let $ G=G_{1}\veedot G_{2}$, $V(G)=V(G_{1})\cup I(G_{1}) \cup V(G_{2})$. The vertices in $V(G_{1})$, $I(G_{1})$, and $V(G_{2})$ are denoted by $v_{j}$, $u_{k}$, and $w_{l}$ respectively.
 If $G_{1}=K_{1}$, then clearly $G^{e}$ is connected.\\
Case 1: $G_{1}$ is not a star graph\\
In $G$, it is easy to check that,
   \begin{align*}
        e(v_{j})=3 , & \text{ for all } v_{j} \in V(G_{1}),\\
        3\leq e(u_{k}) \leq 4, & \text{ for all } u_{k} \in I(G_{1}),\\
        e(w_{l})=2,& \text{ for all } w_{l} \in V(G_{2}).
   \end{align*}
In $G^{e}$, every vertex of $I(G_{1})$ is adjacent to every vertex in $V(G_{2})$ since $d(u_{k},w_{l})=2=min\{e(u_{k}),$ $ e(w_{l})\}.$
Here  $G_{1}$ is not a star graph , so for each vertex $v_{j}$ in $V(G_{1})$, there is a vertex $u_{k}$ in $I(G_{1})$ such that $d(v_{j},u_{k})=3=min\{e(v_{j}), e(u{k})\}$. Therefore, each $v_{j}$ in $V(G_{1})$ is adjacent to  some $u_{k}$ in $I(G_{1})$. Hence $G^{e}$ is connected.\\
Case 2: $G_{1}= K_{1, p_{1}-1}, p_{1}\geq 3$. \\
Let $v_{1}$ be the unique vertex in $V(G_{1})$ such that $d_{eg}(v_{1})=p_{1}-1$. In $G$,  it is easy to see that $e(v_{j})=3,$ for all $v_{j} \in V(G_{1}) \setminus  \{v_{1}\}$ and $e(v_{1})=2$. Also $e(u_{k})=3$ for all $u_{k} \in I(G_{1})$ and $e(w_{l})=2$ , for each $w_{l}$ in $V(G_{2})$.
Since $d(u_{k},w_{l})=2=min\{e(u_{k}),e(w_{l})\}$, in $G^{e}$, every vertex of $V(G_{2})$ is adjacent to every vertex in $I(G_{1})$.
Also, for each vertex $v_{j}$, $j\neq 1$, there exist $p_{1}-2$ vertices in $I(G_{1})$ such that $d(v_{j}, u_{k})=3=min\{e(v_{j}),e(u_{k})\}$.  Moreover, for $j\neq 1$, $d(v_{1}, v_{j})=2=min\{e(v_{1}, e(v_{j})\}$. 
Hence there is a path joining vertices of $V(G_{1})$ and $I(G_{1})$.  
Therefore, $G^{e}$ is a connected graph.
By  Theorem \ref{irreducibilitytheorem}, $\epsilon(G_{1}\veedot G_{2})$ is irreducible.


\end{proof}

\begin{theorem}\label{subdivisionvertex}
Let $G_{j}$, be $r_{j}-$regular ($r_{1}\geq 2$), $(p_{j}, q_{j})$ graph, $j=1, 2$. Let $\{r_{1},\lambda_{2},\ldots, \lambda_{p_{1}}\}$ and  $\{r_{2},\beta_{2},\ldots, \beta_{p_{2}}\}$ be the $A$-eigenvalues of $G_{1}$ and $G_{2}$ respectively. Then the $\epsilon$-eigenvalues of $\epsilon(G_{1}\veedot G_{2})$ consists of
\begin{enumerate}
    \item $4$,   with multiplicity $q_{1}-p_{1}$.
    \item $-2(1+\beta_{k})$,   $k=2, 3,\ldots, p_{2}$.
    \item $-3t_{1}+4-4(\lambda_{i}+r_{1})$,  $i=2,3,\ldots,  p_{1}$, where $t_{1} =\frac{2(1-(\lambda_{i}+r_{1}))+\sqrt{4(\lambda_{i}+r_{1})^{2}+(\lambda_{i}+r_{1})+4}}{3}$.
    \item $-3t_{2}+4-4(\lambda_{i}+r_{1})$,  $i=2,3,\ldots, p_{1}$, where $t_{2} =\frac{2(1-(\lambda+r_{1})-\sqrt{4(\lambda_{i}+r_{1})^{2}+(\lambda_{i}+r_{1})+4}}{3}$.
    \item  3 eigenvalues of the equitable quotient matrix of $\epsilon(G_{1}\veedot G_{2})$, \begin{center}
        
    $\begin{pmatrix}
               0 & 3q_{1}-3r_{1} &0  \\
               3p_{1}-6 & 4q_{1}-8r_{1}+4 &2p_{2} \\
               0& 2q_{1}&2(p_{2}-r_{2}-1)
             \end{pmatrix}$.\\
\end{center}
\end{enumerate}
\end{theorem}
\begin{proof}

     Let $G_{1}$ and $G_{2}$ be regular graphs with regularity $r_{1}(\geq 2)$ and $r_{2}$ respectively. Then, by a appropriate labelling of the vertices of $G_{1}\veedot G_{2}$, its eccentricity matrix can be written as \begin{equation*}
        \epsilon(G_{1}\veedot G_{2}) =\begin{pmatrix}
         0 & 3(J-R)& 0\\
        3(J-R^{T}) & 4(J-I-B)& 2J\\
        0& 2J& 2(J-I-A_{2})
    \end{pmatrix},
                            \end{equation*}
where $R$, $B$ and $A_{2}$ represents the incidence matrix of $G_{1}$, adjacency matrix of $L(G_{1})$ and adjacency matrix of $G_{2}$  respectively.\\
Now, let $Z$ be an eigenvector of $B$ corresponding to the  eigenvalue $-2$ (with multiplicity $q_{1}-p_{1}).$ By Lemma \ref{Beigenvalue} $RZ=0.$ 
Consider $ U=\begin{pmatrix}
        0\\
        Z\\
        0
    \end{pmatrix},$ then we have
                \begin{equation*}
               \epsilon(G_{1}\veedot G_{2}) U=   \begin{pmatrix}
                                           0 & 3(J-R)& 0\\
                                           3(J-R^{T}) & 4(J-I-B)& 2J\\
                                           0& 2J& 2(J-I-A_{2})
                               \end{pmatrix}
                                             \begin{pmatrix}
                                                   0\\
                                                   Z\\
                                                   0
                                            \end{pmatrix}\\
                                   =\begin{pmatrix}
                                       3(J-R)Z\\
                                       4(J-I-B)Z\\
                                       2JZ
                                   \end{pmatrix} \\ 
                                   =4\begin{pmatrix}
                                     0\\
                                     Z\\
                                     0
                                   \end{pmatrix}
                                   =4U.
        \end{equation*}
   Thus, $4$ is an eigenvalue  of $\epsilon(G_{1}\veedot G_{2})$ with multiplicity $q_{1}-p_{1}$.\\
   Let $W$ be an eigenvector of $A_{2}$ corresponding to the eigenvalue $\beta_{k}$ ($k=2,\ldots, p_{2}$). 
   Let $V=\begin{pmatrix}
                     0\\
                     0\\
                     W
                     \end{pmatrix},
                     $ 
                     then we have
   \begin{align*}
                     \epsilon(G_{1}\veedot G_{2})
                     V
                     =\begin{pmatrix}
                                           0 & 3(J-R)& 0\\
                                           3(J-R^{T}) & 4(J-I-B)& 2J\\
                                           0& 2J& 2(J-I-A_{2})
                     \end{pmatrix}
                                             \begin{pmatrix}
                                                   0\\
                                                   0\\
                                                   W
                                            \end{pmatrix}
                                    =-2(1+\beta_{i})\begin{pmatrix}
                                     0\\
                                     0\\
                                     W
                                   \end{pmatrix}
                                   =-2(1+\beta_{i}) V.
        \end{align*}
  Thus,    $-2(1+\beta_{i})$ ( for $i=2, 3,\ldots, p_{2}$)is an eigenvalue of $\epsilon(G_{1}\veedot G_{2})$. \\
 %
  Let $X$ be an eigenvector of $A_{1}$ corresponding to the eigenvalue $\lambda_{i}$ ( for $i=2, \ldots, p_{1}$), where $A_{1}$ is the adjacency matrix of $G_{1}$. Using   Lemma \ref{thelinegraph}, we get $R^{T}X$ is an eigenvector of $B$ corresponding to the eigenvalue $\lambda_{i}+r_{1}-2.$ Thus the vectors $X$ and       $R^{T}X$ are orthogonal to the vector $J_{{p_{1}}\times 1}.$\\
  Now, consider the vector $\phi=\begin{pmatrix}
                                        tX\\
                                          R^{T}X\\
                                           0   
                                    \end{pmatrix}.$       
  Next, we determine under what condition $\phi$ is an eigenvector of   $ \epsilon(G_{1} \veedot G_{2}).$
  If $\mu$ is an eigenvalue of $\epsilon(G_{1}\veedot G_{2})$ such that 
                              $\epsilon(G_{1}\veedot G_{2})\phi=\mu \phi$, then
\begin{align*}
                                        \begin{pmatrix}
                                           0 & 3(J-R)& 0\\
                                           3(J-R^{T}) & 4(J-I-B)& 2J\\
                                           0& 2J& 2(J-I-A_{2})
                                          \end{pmatrix}
                                          \begin{pmatrix}
                                        tX\\
                                          R^{T}X\\
                                           0  
                                           \end{pmatrix}
                                           &= \mu \begin{pmatrix}
                                        tX\\
                                          R^{T}X\\
                                           0  
                                           \end{pmatrix}\\
                                           \begin{pmatrix}
                                               -3(\lambda_{i}+r_{1})X\\
                                            (-3t+4-4(\lambda_{i}+r_{1}))R^{T}X\\
                                            0
                                           \end{pmatrix}
                                           &=\mu \begin{pmatrix}
                                              tX\\
                                          R^{T}X\\
                                           0  
                                           \end{pmatrix}
\end{align*}

      Therefore,                      \begin{align*}
                                  -3(\lambda_{i}+r_{1})&=\mu t \\
                                  -3t+4-4(\lambda_{i}+r_{1})&= \mu\\
                             \end{align*}
    From this we  will get
                             \begin{equation*}
                             3t^{2}+4t((\lambda_{i}+r_{1})-1)-3(\lambda_{i}+r_{1})=0
                                            \end{equation*}
     So $t$ has two values \begin{align*}
                             t_{1} &=\frac{2(1-(\lambda_{i}+r_{1}))+\sqrt{4(\lambda_{i}+r_{1})^{2}+(\lambda_{i}+r_{1})+4}}{3}.\\
                             t_{2} &=\frac{2(1-(\lambda+r_{1})-\sqrt{4(\lambda_{i}+r_{1})^{2}+(\lambda_{i}+r_{1})+4}}{3}.
                           \end{align*}  
 Therefore, corresponding to each eigenvalue $\lambda_{i}$ ( for $i=2,\ldots, p_{1}$) of $G_{1}$, we get $2$ eigenvalues for $\epsilon(G_{1}\veedot G_{2}),$ as $-3t_{1}+4-4(\lambda_{i}+r_{1})$ and   $-3t_{2}+4-4(\lambda_{i}+r_{1})$. Thus, we get $2(p_{1}-1)$ eigenvalues of $\epsilon(G_{1}\veedot G_{2}).$\\       
 The remaining $3$ eigenvalues of $\epsilon(G_{1}\veedot G_{2})$ are eigenvalues of its equitable quotient matrix, 
                   \begin{center}
                   $F=\begin{pmatrix}
               0 & 3q_{1}-3r_{1} &0  \\
               3p_{1}-6 & 4q_{1}-8r_{1}+4 &2p_{2} \\
               0& 2q_{1}&2(p_{2}-r_{2}-1)
             \end{pmatrix}$.\\
\end{center}
This completes the proof.
\end{proof}
        
        
    
As an application we construct new pairs of $\epsilon-$cospectral graphs in the following theorem.
\begin{corollary}
   Let $G_{1}$ and $G_{2}$ be two regular cospectral graphs and $H$ be any arbitrary regular graph. 
   Then\begin{enumerate}
       \item $G_{1}\veedot H$ and $G_{2}\veedot H$  are 
       $\epsilon-$cospectral. 
   \item $H \veedot G_{1}$ and $H \veedot G_{2}$ are $\epsilon-$ cospectral.
     \end{enumerate}
\end{corollary}

\begin{corollary}\label{wienervertexjoin}
    Let $G_{j}$ be an $r_{j}-$regular $(r_{1}\geq 2)$, $(p_{j}, q_{j})$ graph, $j=1, 2$. Then the $eccentricity$ $Wiener$ $index$ of $\epsilon(G_{1}\veedot G_{2})$ is given by,\begin{equation*}
        W_{\epsilon}(G_{1}\veedot G_{2})=q_{1}(3p_{1}-4r_{1}+2q_{1}-1)+p_{2}(p_{2}+2q_{1}-r_{2}-1)-\frac{3}{2}r_{1}p_{1}.
    \end{equation*}

\end{corollary}
\begin{corollary}
    Let $G_{j}$ be an $r_{j}-$regular $(r_{1} \geq 2)$, $(p_{j}, q_{j})$ graph, $j= 1, 2$. Then\begin{equation*}
        \rho_{\epsilon}(G_{1}\veedot G_{2}) > \frac{2q_{1}(3p_{1}-4r_{1}+2q_{1}-1)+2p_{2}(p_{2}+2q_{1}-r_{2}-1)-3p_{1}r_{1}}{p_{1}+p_{2}+q_{1}}.
    \end{equation*}
\end{corollary}
\begin{proof}
 Proof follows from Theorem \ref{spectralradiuslowerbound} and corollary \ref{wienervertexjoin}  
\end{proof}

\section{Irreducibility and eccentricity spectrum of subdivision edge join of two graphs}
\begin{theorem}
    Let $G_{i}$ be a $(p_{i}, q_{i})$ graph, $i=1, 2$. Then $\epsilon (G_{1}\veebar G_{2})$ is irreducible.
    \end{theorem}
    \begin{proof}

 Let $V(G_{1}\veebar G_{2})=V(G_{1})\cup I(G_{1}) \cup V(G_{2})$ be a partition of the vertex set of $G=G_{1}\veebar G_{2}$. Let $v_{l}$, $u_{j}$ and $w_{k}$ represents the vertices in $V(G_{1}), I(G_{1})$ and $V(G_{2})$ respectively.\\
Case 1: $G_{1}$ is not a star graph.\\
In $G_{1}\veebar G_{2}$ we get,
          \begin{align*}
         3\leq e(v_{l})\leq 4, &\text{ for each }  v_{l} \in V(G_{1})\\
         e(u_{j})=3 &\text{ for each } u_{j} \in I(G_{1})\\
          e(w_{k})=2, &\text{ for each }w_{k}\in V(G_{2})\\
     \end{align*}
Since $d(w_{k}, v_{l})=2=min\{e(w_{k}),e(v_{l})\}$,  every vertex in $V(G_{2})$ is adjacent to every vertex in $V(G_{1})$ in $G^{e}$. Also for each vertex $u_{j}\in I(G_{1})$, the eccentric vertex of $u_{j}$ will be some $v_{l}$. Hence in $G^{e}$ every vertex in $I(G_{1})$ will be adjacent to some vertex in $V(G_{1})$. Therefore  $G^{e}$ is connected.\\
Case 2: $G_{1}$ is a star graph, $G_{1}=K_{1,p_{1}-1}$\\
If $p_{1}=2$, clearly $(G_{1}\veebar G_{2})^{e}$ is connected.\\
 Let $v_{1}$ be the unique vertex in $G_{1}$ such that $deg(v_{1})=p_{1}-1$, $p_{1}>2,$ then \begin{align*}
e(v_{1})=2 \text{ and } e(v_{l})=4, &\text{ for } l\neq 1\\
e(u_{j})=3, &\text{ for } u_{j}\in I(G_{1})\\
e(w_{k})=2, &\text{ for } w_{k}\in V(G_{2}). 
\end{align*}
$d(v_{l},w_{k})=2=min\{e(v_{l}),e(w_{k})\}$, so every vertex in $V(G_{1})$ is adjacent to every vertex in $V(G_{2})$. And for each $v_{l}, l\neq 1$ there exist $p_{1}-2$ vertices $u_{j}$ such that $d(v_{l},u_{j})=3=min\{e(v_{l}),e(u_{j})\}$.  Therefore $G^{e}$ is connected.
Hence by Theorem \ref{irreducibilitytheorem} $\epsilon(G_{1}\veebar G_{2})$ is irreducible.
\end{proof}

 \begin{theorem}\label{subdivisionedge}
Let $G_{j}$ be $r_{j}-$regular ($r_{1}\geq 2$), $(p_{j},q_{j})$ graph, $j=1, 2$. Let
$\{r_{1},\lambda_{2},\ldots, \lambda_{p_{1}}\}$ and 
$\{r_{2},\gamma_{2},\ldots,  \gamma_{p_{2}}\}$ be the $A$-eigenvalues of $G_{1}$ and $G_{2}$ respectively. Then the $\epsilon$-eigenvalues of $\epsilon(G_{1}\veebar G_{2})$ consists of
\begin{enumerate}
    \item $0$,   with multiplicity $q_{1}-p_{1}$.
    \item $-2(1+\gamma_{k})$,   $k=2, 3,\ldots, p_{2}$.
    \item$-2(1+\lambda_{i})\pm\sqrt{4(1+\lambda_{i})^2+9(\lambda_{i}+r_{1})}$, $i=2, \ldots, p_{1}.$
    \item  3 eigenvalues of the equitable quotient matrix of $\epsilon(G_{1}\veebar G_{2})$, \begin{center}
            $\begin{pmatrix}
               4(p_{1}-1-r_{1}) & 3(q_{1}-r_{1}) &2p_{2}  \\
               3(p_{1}-2) & 0 & 0 \\
               2p_{1}& 0&2(p_{2}-r_{2}-1)
             \end{pmatrix}$.\\
\end{center}
\end{enumerate}
\end{theorem}
\begin{proof}

 By a proper labeling  of the vertices of $G_{1}\veebar G_{2}$, we get
    \begin{equation*} 
            \epsilon(G_{1}\veebar G_{2}) =\begin{pmatrix}
         4(J-I-A_{1})& 3(J-R)& 2J\\
        3(J-R^{T}) & 0& 0\\
        2J& 0& 2(J-I-A_{2})
        \end{pmatrix},
           \end{equation*}                  
       where $R$ and $A_{j}$ represents the incidence matrix of $G_{1}$, adjacency  matrix of $G_{j}$ for $j=1,2.$\\
 Let $Z$ be an eigenvector of $B$ corresponding to the eigenvalue $-2$ (with multiplicity $q_{1}-p_{1}$), where $B$ is the adjacency matrix of $L(G_{1})$. By Lemma \ref{Beigenvalue}, $RZ=0$. Consider $\psi=\begin{pmatrix}
                                     0\\
                                     Z\\
                                     0
                                 \end{pmatrix}$.
        Then we have \begin{align*}    
        \epsilon(G_{1}\veebar G_{2})\psi=\begin{pmatrix}
                                                   4(J-I-A_{1})&3(J-R)&2J\\
                                                   3(J-R^{T})&0&0\\
                                                   2J&0&2(J-I-A_{2})
                                                  \end{pmatrix}
                                                  \begin{pmatrix}
                                                      0\\
                                                      Z\\
                                                      0
                                                  \end{pmatrix}
                                             =0\begin{pmatrix}
                                                 0\\
                                                 Z\\
                                                 0
                                             \end{pmatrix}
                                            =0 \psi.
                                            \end{align*}
                           
Thus, $0$ is an eigenvalue of $\epsilon(G_{1}\veebar G_{2})$ with multiplicity $q_{1}-p_{1}$.\\
Let $W$ be an eigenvector of $A_{2}$ corresponding to the eigenvalue  $\gamma_{k}$ ( for $k=2,\ldots, p_{2}$).                                                              Let $\zeta=\begin{pmatrix}
                                                0\\
                                               0\\
                                                W
                                            \end{pmatrix},$
   then we have $\epsilon(G_{1}\veebar G_{2}) \zeta=-2(1+\gamma_{k}) \zeta$. Thus, $-2(1+\gamma_{k})$ ( for $k={2,\ldots, p_{2}}$) is an eigenvalue of $\epsilon(G_{1}\veebar G_{2}).$

Let $X$ be an eigenvector of $A_{1}$ corresponding to the eigenvalue $\lambda_{i}$ ( for $i=2,\ldots,p_{1}$).  Using Lemma \ref{thelinegraph}, we get $R^{T}X$ is an eigenvector of $B$ corresponding to the eigenvalue $\lambda_{i}+r_{1}-2.$ 
Thus, the vectors $X$ and $R^{T}X$ are orthogonal to the vector $J_{p_{1}\times 1}.$\\ 
Now, consider the vector $\phi^{*}=\begin{pmatrix}
    tX\\
    R^{T}X\\
    0
\end{pmatrix}.$
Next, we determine under what condition $\phi^{*}$ is an eigenvector of $\epsilon(G_{1}\veebar G_{2}).$ If $\mu^{*}$ is an eigenvalue of $\epsilon(G_{1}\veebar G_{2})$ such that $\epsilon(G_{1}\veebar G_{2})\phi^{*}=\mu^{*} \phi^{*}$, then
   \begin{align*}
       \begin{pmatrix}
               4(J-I-A_{1})& 3(J-R)& 2J\\
        3(J-R^{T}) & 0& 0\\
        2J& 0& 2(J-I-A_{2})
             \end{pmatrix}
             \begin{pmatrix}
       tX\\
       R^{T}X\\
       0
   \end{pmatrix}
   = \mu^{*} \begin{pmatrix}
      tX\\
      R^{T}X\\
      0
   \end{pmatrix}\\
  \begin{pmatrix}
            -4t(1+\lambda_{i})X-3(\lambda_{i}+r_{1})X\\
            -3tR^{T}X\\
            0
        \end{pmatrix}
        = \mu^{*} \begin{pmatrix}
            tX\\
            R^{T}X\\
            0
        \end{pmatrix}.    
   \end{align*}
   \text{ Therefore }
   \begin{align*}
       -(4t+4t\lambda_{i}+3\lambda_{i}+3r_{1})&=\mu^{*} t\\
        -3t&=\mu^{*}.
   \end{align*}
From this we will get \begin{equation*}
    3t^{2}-4(1+\lambda_{i})t-3(\lambda_{i}+r_{1})=0
                        \end{equation*} 
 so that $t$ has two values
\begin{equation*}
     t_{1} =\frac{2(1+\lambda_{i})+\sqrt{4(1+\lambda_{i})^{2}+9(\lambda_{i}+r_{1})}}{3}, 
     t_{2} =\frac{2(1+\lambda_{i})-\sqrt{4(1+\lambda_{i})^{2}+9(\lambda_{i}+r_{1})}}{3}.
     \end{equation*}
   Therefore, corresponding to each eigenvalue $\lambda_{i}$  ( for $i=2,\ldots,p_{1}$) of $G_{1}$ we get 2 eigenvalues of $\epsilon(G_{1}\veebar G_{2})$ as $-3t_{1}$ and $-3t_{2}$. Hence, we get $2(p_{1}-1)$ eigenvalues of $\epsilon(G_{1}\veebar G_{2})$.\\
   The remaining $3$ eigenvalues of $\epsilon(G_{1}\veebar G_{2})$ are eigenvalues of its equitable  quotient matrix, 
   \begin{center}    
                 $F=\begin{pmatrix}
                  4(p_{1}-1-r_{1}) & 3(q_{1}-r_{1}) &2p_{2}  \\
                  3(p_{1}-2) & 0 & 0 \\
                  2p_{1}& 0&2(p_{2}-r_{2}-1)
             \end{pmatrix}$.\\                                    
           \end{center}                                  
  \end{proof} 
  The following corollary gives new pairs of $\epsilon-$cospectral graphs.
\begin{corollary}
     Let $G_{1}$ and $G_{2}$ be two regular cospectral graphs and $H$ be any arbitrary regular graph. Then \begin{enumerate}
    \item  $G_{1}\veebar H$ and $G_{2}\veebar H$  are $\epsilon-$ cospectral.
    \item $H \veebar G_{1}$ and $H\veebar G_{2}$ are $\epsilon-$ cospectral.
    \end{enumerate}
\end{corollary}

  \begin{corollary}\label{wienerofedgejoin}
      Let $G_{j}$, $j=1, 2$ be a $r_{j}-$regular ($r_{1}\geq 2$), $(p_{i}, q_{i})$ graph. Then the $eccentricity$ $Wiener$ $index$ of $\epsilon(G_{1}\veebar G_{2})$ is given by, \begin{equation*}
          W_{\epsilon}(G_{1}\veebar G_{2})=2p_{1}^{2}+p_{2}^{2}-\frac{p_{1}}{2}(4-4p_{2}+7r_{1})+3q_{1}(p_{1}-1)-p_{2}(r_{2}+1).
      \end{equation*}  
  \end{corollary}
\begin{corollary}
    Let $G_{j}$ be an $r_{j}-$regular ($r_{1}\geq 2)$, ($p_{j}$, $q_{j}$) graph, $j=1, 2$. Then \begin{equation*}
        \rho_{\epsilon}(G_{1}\veebar G_{2})> \frac{4p_{1}^{2}+2p_{2}^{2}-p_{1}(4-4p_{2}+7r_{1})+6q_{1}(p_{1}-1)-2p_{2}(r_{2}+1)}{p_{1}+p_{2}+q_{1}}.
    \end{equation*}
\end{corollary}

The $join$ $G_{1}\vee G_{2}$ of two graphs $G_{1}$ and $G_{2}$, is the graph obtained from $G_{1}\cup G_{2}$ by joining every vertex of $G_{1}$ with every vertex of $G_{2}$, where $G_{1}\cup G_{2}$ denotes the disjoint union of $G_{1}$ and $G_{2}$. 
In \cite{wang2018anti} the authors established that, if $G$ is a $r-$regular graph of order $p$ and diameter $2$, then
   \begin{equation*}
       \epsilon(G\vee K_{1})=\begin{pmatrix}
           2(J-I-A)&J\\
           J&0
       \end{pmatrix},
   \end{equation*}
   where $A$ is the adjacency matrix of $G$. The same matrix is obtained if $G$ is a $r-$ regular graph with order $p$ and $diam(G) \geq 2$. 
In the next theorem, we consider the $join$ of a non-complete regular graph with $diam(G) \geq 2$ and $K_{1}$. 
\begin{theorem}\label{GjoinK1}
    Let $G$ be a $r-$regular, non-complete, ($p, q$) graph. Then \begin{equation*}
    spec_{\epsilon}(G\vee K_{1})=\begin{pmatrix}
        (p-r-1)\pm \sqrt{(p-r-1)^{2}+p}& -2(\lambda_{2}+1)&\cdots&-2(\lambda_{p}+1)\\
        1&1&\cdots&1
    \end{pmatrix},
\end{equation*}
where $\{r, \lambda_{2}, \lambda_{3},\ldots,  \lambda_{p}\}$ are $A-$eigenvalues of $G$.
\end{theorem}
\begin{proof}
By a proper labeling of vertices in $G\vee K_{1}$, we get $\epsilon(G\vee K_{1})=\begin{pmatrix}
                                      2(J-I-A)&J\\
                                      J&0
                                   \end{pmatrix},$
   where $A$ is the adjacency matrix of $G$. 
Let $X$ be an eigenvector of $A$ corresponding to the eigenvalue $\lambda_{i}$ (for $i=2,\ldots, p$) . Then we have \begin{equation*}
\epsilon(G\vee K_{1})\begin{pmatrix}
    X\\
    0
\end{pmatrix}=-2(1+\lambda_{i})\begin{pmatrix}
    X\\
    0
    \end{pmatrix}.
\end{equation*}
Thus, $-2(1+\lambda_{i})$ (for $i=2,\ldots, p$) is an eigenvalue of $\epsilon(G\vee K_{1})$ .
The remaining $2$ eigenvalues are given by the equitable quotient matrix of $\epsilon(G\vee K_{1})$, $\begin{pmatrix}
    2(p-r-1)&1\\
    p&0
\end{pmatrix}$.

\end{proof}
In \cite{wang2018anti}, the authors discussed  the $\epsilon-$spectrum of $G\vee G$, where $G$ is a regular graph with $diam(G)=2$. Later in \cite{mahato2020spectra}, Mahato et al. discussed  the $\epsilon-$ spectrum of $G_{1}\vee G_{2}$, where $G_{1}$ and $G_{2}$ are non-complete connected graphs. 
\\
In this context, we will refer to the  subsequent theorem for its application in Section $5.$
\begin{theorem}\label{GjoinG}\cite{mahato2020spectra}
    Let $G$ is a $r-$regular, non complete, $(p, q)$ graph. Then the $\epsilon-$spectrum of $G\vee G$ is given by, \begin{equation*}
    Spec_{\epsilon}(G\vee G)=\begin{pmatrix}
       2(p-r-1)&-2(\lambda_{2}+1)&\cdots&-2(\lambda_{p}+1)\\
       2&2&\cdots&2
   \end{pmatrix}, 
\end{equation*}
where $\lambda_{i}\neq r$, $i=2, 3,\ldots, p$ are the $A-$eigenvalues of $G$.
\end{theorem}
Next, we will discuss the $join$ and $subdivision$ $vertex$($edge$)$join$ of a connected graph with a non-connected graph having two components. We will explore its $\epsilon-$ spectrum and reveal some new graph families with irreducible or reducible eccentricity matrix. 
\begin{theorem}\label{joinunion}
    Let $G_{j}$ be an $r_{j}-$regular, $(p_{j}, q_{j})$ graph,  $j=0, 1, 2$, and $diam(G_{1})\geq2$. Let  $\{r_{0}, \lambda_{2}, \lambda_{3},$ $\ldots,  \lambda_{p_{0}}\}$,
     $\{r_{1}, \beta_{2}, \beta_{3},\ldots,  \beta_{p_{1}}\}$ and  $\{r_{2}, \gamma_{2}, \gamma_{3},\ldots,  \gamma_{p_{2}}\}$ are the $A-$ eigenvalues of $G_{0}$, $G_{1}$ and $G_{2}$ respectively. Then $\epsilon-$ eigenvalues of $G_{0}\vee (G_{1}\cup G_{2})$ consists of
     \begin{enumerate}
         \item $-2(1+\lambda_{i})$, $i=2,\ldots, p_{0}.$
         \item $-2(1+\beta_{k})$, $k=2,\ldots, p_{1}.$
         \item $-2(1+\gamma_{l})$, $l=2,\ldots,  p_{2}.$
         \item $3$ eigenvalues of the equitable quotient matrix of $\epsilon(G_{0}\vee (G_{1}\cup G_{2}))$,\begin{center}
             $\begin{pmatrix}
                 2(p_{0}-1-r_{0})&0&0\\
                 0&2(p_{1}-1-r_{1})&2p_{2}\\
                 0&2p_{1}&2(p_{2}-1-r_{2})
             \end{pmatrix}$.
         \end{center} 
     \end{enumerate}
\end{theorem}
\begin{proof}
    By a proper labeling for the vertices of $G_{0}\vee (G_{1}\cup G_{2})$, we get
    \begin{center}
      $\epsilon( G_{0}\vee (G_{1}\cup G_{2}))=\begin{pmatrix}
          2(J-I-A_{0})&0 &0\\
          0&2(J-I-A_{1})&2J\\
          0&2J&2(J-I-A_{2})
        \end{pmatrix}$,
    \end{center}
    where $A_{j}$ denotes the adjacency matrix of $G_{j}$ for $j=0, 1, 2$.\\
    Let $X$ be an eigenvector of $A_{0}$ corresponding to the eigenvalue $\lambda_{i}$ (for $i=2,\ldots, p_{0}$), $Y$ be an eigenvector of $A_{1}$ corresponding to the eigenvalue $\beta_{k}$ (for $k=2,\ldots, p_{1}$), and $Z$ be an eigenvector of $A_{2}$ corresponding to the eigenvalue $\gamma_{l}$ (for $l=2,\ldots, p_{2}$). Then we have, 
        $\epsilon( G_{0}\vee (G_{1}\cup G_{2}))
        \begin{pmatrix}
            X\\
            0\\
            0
        \end{pmatrix}=
        -2(1+\lambda_{i})\begin{pmatrix}
            X\\
            0\\
            0
        \end{pmatrix},$
        $\epsilon( G_{0}\vee (G_{1}\cup G_{2}))
        \begin{pmatrix}
            0\\
            Y\\
            0
        \end{pmatrix}=
        -2(1+\beta_{k})\begin{pmatrix}
            0\\
            Y\\
            0
        \end{pmatrix},$
 and  
      $ \epsilon( G_{0}\vee (G_{1}\cup G_{2}))
        \begin{pmatrix}
            0\\
            0\\
            Z
        \end{pmatrix}=
        -2(1+\gamma_{l})\begin{pmatrix}
            0\\
            0\\
            Z
        \end{pmatrix}.$
The remaining $3$ eigenvalues of $\epsilon( G_{0}\vee (G_{1}\cup G_{2}))$ are given by its equitable quotient matrix, 
\begin{center}
             $F=\begin{pmatrix}
                 2(p_{0}-1-r_{0})&0&0\\
                 0&2(p_{1}-1-r_{1})&2p_{2}\\
                 0&2p_{1}&2(p_{2}-1-r_{2}
             \end{pmatrix}$.
         \end{center}  
\end{proof}
\begin{theorem}
    Let $G_{0}$, $G_{1}$ and $G_{2}$ be any three graphs. The following gives some new families of graphs having either a reducible or irreducible eccentricity matrix.
    \begin{enumerate}
        \item If $G_{0}$ is self -centered and $diam(G_{0})=2$, then $G_{0}\vee (G_{1}\cup G_{2})$ is reducible.
        \item If $G_{0}$ is not self-centered and $diam(G_{0})=2$, then $G_{0}\vee (G_{1}\cup G_{2})$ is irreducible.
        \item If $diam(G_{0})\geq 3$, then  $G_{0}\vee (G_{1}\cup G_{2})$ is reducible.
    \end{enumerate}
\end{theorem}
\begin{proof}
    Let $G=G_{0}\vee (G_{1}\cup G_{2})$, $V(G)=V(G_{0})\cup V(G_{1})\cup V(G_{2})$ and let $u_{i}$, $v_{j}$, $w_{k}$ be the  vertices in $V(G_{0})$, $V(G_{1})$, and  $V(G_{2})$ respectively. 
    In $G$, \begin{align*}
        e(v_{j})=2, &\text{ for every } v_{j} \in V(G_{1})\\
        e(w_{k})=2, &\text{ for every } w_{k}\in V(G_{2})
        \end{align*}
and $d(v_{j},w_{k})=2.$
Therefore, every vertex in $V(G_{1})$ is adjacent to every vertex of $V(G_{2})$ in $G^{e}$.
\begin{enumerate}
    \item  If $G_{0}$ is self-centered and $diam(G_{0})=2$, clearly in $G$ for each $u_{i}$ in $V(G_{0})$ , $e(u_{i})=2$. But         \begin{align*}
   d(u_{i},v_{j})&=1\neq min\{e(u_{i}), e(v_{j})\} \\
    d(u_{i},w_{k})&=1\neq min\{e(u_{i}), e(w_{k})\}.
    \end{align*}
    Therefore, 
    $G^{e}$ is a disconnected graph. Thus, the result follows from Theorem \ref{irreducibilitytheorem}.
    \item If $G_{0}$ is not self-centered and $diam(G)=2$, then there exist at least one $u_{l}$ in $V(G_{0})$ such that $e(u_{l})=1$. 
    Then we have 
    \begin{align*}
        d(u_{l}, u_{i})&=1=min\{e(u_{l}),e(u_{i})\},\text{ for all } u_{i} \text{ in } V(G_{0}),\\
        d(u_{l}, v_{j})&=1=min\{e(u_{l}),e(v_{j})\},\text{ for all } v_{j} \text{ in } V(G_{1}),\\
        d(u_{l}, w_{k})&=1=min\{e(u_{l}),e(w_{k})\},\text{ for all } w_{k} \text{ in } V(G_{2}).
    \end{align*}
    Thus, the eccentric graph $G^{e}$ is connected. Therefore, $\epsilon(G)$ is irreducible by Theorem \ref{irreducibilitytheorem}.
    
   \item  If $diam(G_{0})\geq 3,$ then in $G$, $e(u_{i})=2$ , for every $u_{i}$ in $V(G_{0})$. So in $G^{e}$, no vertex in $V(G_{0})$ is adjacent to any vertex in $V(G_{1})$ and $V(G_{2})$. Thus, $G^{e}$ is disconnected. Hence,  using 
   Theorem \ref{irreducibilitytheorem}, $\epsilon(G)$ is reducible.
\end{enumerate}
\end{proof}
\begin{theorem}\label{subdivisionvertexunion}
    Let $G_{j}$ be an $r_{j}-$regular ($r_{0}\geq 2$), $(p_{j}, q_{j})$ graph, $j=0,1,2$ and let $\{r_{0},\lambda_{2},\ldots, \lambda_{p_{0}}\}$,
    $\{r_{1}, \beta_{2}, \beta_{3},\ldots, \beta_{p_{1}}\}$
    and  $\{r_{2},\gamma_{2},\gamma_{3},\ldots, \gamma_{p_{2}}\}$ be the $A-$ eigenvalues of $G_{0}$, $G_{1}$ and $G_{2}$ respectively. Then $\epsilon-$ eigenvalues of $G_{0}\veedot (G_{1}\cup G_{2})$ consists of\begin{enumerate}
        \item $4$  with multiplicity $q_{0}-p_{0}.$
        \item $-2(1+\beta_{k}), k=2,\ldots,  p_{1}.$
        \item $-2(1+\gamma_{l}), l=2,\ldots,  p_{2}.$
        \item $-3t_{1}-4-4(\lambda_{i}+r_{0}-2),i=2,\ldots, p_{0}, t_{1}=\frac{2(1-(\lambda_{i}+r_{0}))+\sqrt{4(1-(\lambda_{i}+r_{0}))^{2}+9(\lambda_{i}+r_{0})}}{3}.$
        \item $-3t_{2}-4-4(\lambda_{i}+r_{0}-2),i=2,\ldots,  p_{0}, t_{2}=\frac{2(1-(\lambda_{i}+r_{0}))-\sqrt{4(1-(\lambda_{i}+r_{0}))^{2}+9(\lambda_{i}+r_{0})}}{3}.$
        \item $4$ eigenvalues of the equitable  quotient matrix of $\epsilon(G_{0}\veedot (G_{1}\cup G_{2}))$, \begin{center}
        $\begin{pmatrix}
            0&3(q_{0}-r_{0})&0&0\\
            3(p_{0}-2)&4(q_{0}-2r_{0}+1)&2p_{1}&2p_{2}\\
            0&2q_{0}&2(p_{1}-r_{1}-1)&2p_{2}\\
            0&2q_{0}&2p_{1}&2(p_{2}-r_{2}-1)\\
        \end{pmatrix}.$
        \end{center}
    \end{enumerate}
\end{theorem}
\begin{proof}
    By a proper ordering of the vertices of $G_{0}\veedot (G_{1}\cup G_{2})$ ,\begin{equation*}
         \epsilon(G_{0}\veedot (G_{1}\cup G_{2}))=\begin{pmatrix}
               0 & 3(J-R_{0})& 0&0\\
        3(J-R_{0}^{T}) & 4(J-I-B_{0})& 2J&2J\\
        0& 2J& 2(J-I-A_{1})&2J\\
        0&2J&2J&2(J-I-A_{2})
         \end{pmatrix},
    \end{equation*}
    where $R_{0}$ represents the incidence matrix of $G_{0}$ and $B_{0}$, $A_{1}$, $A_{2}$ represents the  adjacency matrix of $L(G_{0})$, $G_{1}$, $G_{2}$ respectively.
    Let $W$ be an eigenvector of $A_{1}$ corresponding to the eigenvalue $\beta_{k}$ (for $k=2, \ldots, p_{1}$), $V$ be an eigenvector of $A_{2}$ corresponding to the eigenvalue $\gamma_{l}$ (for $l=2, \ldots, p_{2}$), and $X$ be the eigenvector of $B_{0}$ corresponding to the eigenvalues $-2$ (with multiplicity $q_{0}-p_{0}$). Then we have
     \begin{align*}
       \begin{pmatrix}
               0 & 3(J-R_{0})& 0&0\\
        3(J-R_{0}^{T}) & 4(J-I-B_{0})& 2J&2J\\
        0& 2J& 2(J-I-A_{1})&2J\\
        0&2J&2J&2(J-I-A_{2})
         \end{pmatrix}
         \begin{pmatrix}
             0\\
             0\\
             W\\
             0
         \end{pmatrix}&=-2(1+\beta_{k})\begin{pmatrix}
             0\\
             0\\
             W\\
             0
         \end{pmatrix},  \\ 
         \begin{pmatrix}
               0 & 3(J-R_{0})& 0&0\\
        3(J-R_{0}^{T}) & 4(J-I-B_{0})& 2J&2J\\
        0& 2J& 2(J-I-A_{1})&2J\\
        0&2J&2J&2(J-I-A_{2})
         \end{pmatrix}
         \begin{pmatrix}
             0\\
             0\\
             0\\
             V
         \end{pmatrix}&=-2(1+\gamma_{l})\begin{pmatrix}
             0\\
             0\\
             0\\
             V
         \end{pmatrix},\\
         \begin{pmatrix}
               0 & 3(J-R_{0})& 0&0\\
        3(J-R_{0}^{T}) & 4(J-I-B_{0})& 2J&2J\\
        0& 2J& 2(J-I-A_{1})&2J\\
        0&2J&2J&2(J-I-A_{2})
         \end{pmatrix}
         \begin{pmatrix}
             0\\
             X\\
             0\\
             0
         \end{pmatrix}&=4\begin{pmatrix}
             0\\
             X\\
             0\\
             0
         \end{pmatrix}. 
     \end{align*}

    Let $U$ be an eigenvector corresponding to the $A-$ eigenvalue $\lambda_{i}$ (for $i=2,\ldots, p_{0}$) of $G_{0}$.
    Using Lemma \ref{thelinegraph}, we get $R_{0}^{T} U$ is an eigenvector of $B_{0}$ corresponding to the eigenvalue $\lambda_{i}+r_{0}-2.$ Thus, the vectors $U$ and $R_{0}^{T}U$ are orthogonal to the vector $J_{p_{0}\times 1}.$\\ 
    Now, consider the vector $\Tilde{\phi}=\begin{pmatrix}
             tU\\
             R_{0}^{T}U\\
             0\\
             0
         \end{pmatrix}$.
         Next, we determine under what condition $\Tilde{\phi}$ is an eigenvector of  $\epsilon(G_{0}\veedot (G_{1}\cup G_{2}))$. If $\tilde{\mu}$ is an eigenvalue of $\epsilon(G_{0}\veedot (G_{1}\cup G_{2}))$ such that $\epsilon(G_{0}\veedot (G_{1}\cup G_{2}))\Tilde{\phi}=\Tilde{\mu} \Tilde{\phi}$, then
    \begin{equation*}
        \begin{pmatrix}
               0 & 3(J-R_{0})& 0&0\\
        3(J-R_{0}^{T}) & 4(J-I-B_{0})& 2J&2J\\
        0& 2J& 2(J-I-A_{1})&2J\\
        0&2J&2J&2(J-I-A_{2})
         \end{pmatrix}
         \begin{pmatrix}
             tU\\
             R_{0}^{T}U\\
             0\\
             0
         \end{pmatrix}=\Tilde{\mu}\begin{pmatrix}
             tU\\
             R_{0}^{T}U\\
             0\\
             0
         \end{pmatrix}.       
    \end{equation*} 
  Therefore,
  \begin{align*}
     -3(\lambda_{i}+r_{0})&=\Tilde{\mu} t\\
     -3t-4-4(\lambda_{i}+r_{0}-2)&=\Tilde{\mu}.
 \end{align*}   
 From this we will get,
 \begin{equation*}
 \Tilde{\mu}=-3t_{1}-4-4(\lambda_{i}+r_{0}-2),i=2,\ldots,  p_{0},
 \text{ where }
 t_{1}=\frac{2(1-(\lambda_{i}+r_{0}))+\sqrt{4(1-(\lambda_{i}+r_{0}))^{2}+9(\lambda_{i}+r_{0})}}{3} 
 \end{equation*}
 and\begin{equation*} 
  \Tilde{\mu}=  -3t_{2}-4-4(\lambda_{i}+r_{0}-2),i=2,\ldots p_{0},
     \text{ where }\\ t_{2}=\frac{2(1-(\lambda_{i}+r_{0}))-\sqrt{4(1-(\lambda_{i}+r_{0}))^{2}+9(\lambda_{i}+r_{0})}}{3}.
 \end{equation*}
 And the remaining $4$ eigenvalues of $\epsilon(G_{0}\veedot (G_{1}\cup G_{2}))$ are given by its equitable quotient matrix, 
 \begin{center}
     $F=\begin{pmatrix}
            0&3(q_{0}-r_{0})&0&0\\
            3(p_{0}-2)&4(q_{0}-2r_{0}+1)&2p_{1}&2p_{2}\\
            0&2q_{0}&2(p_{1}-r_{1}-1)&2p_{2}\\
            0&2q_{0}&2p_{1}&2(p_{2}-r_{2}-1)\\
        \end{pmatrix}$.
 \end{center}
\end{proof}
The following remarks describes a family of graphs with an irreducible eccentricity matrix.
\begin{remark}
    For any three connected graphs $G_{0}$, $G_{1}$ and $G_{2}$ the matrix $\epsilon(G_{0}\veedot (G_{1}\cup G_{2}))$ is  irreducible, unless $G_{0}=K_{1,1}$. 
\end{remark}
\begin{remark}
    For any connected graphs $G_{i}$,  $i=0, 1, \ldots, p$, the matrix  $\epsilon(G_{0}\veedot (\cup^{p}_{l=1}G_{l})$ is irreducible, unless $G_{0}=K_{1,1}$.
\end{remark}


\begin{theorem}\label{subdivisionedgeunion}
     Let $G_{j}$ be an $r_{j}-$regular ($r_{0}\geq2$),  $(p_{j}, q_{j})$ graph, $j=0,1,2$, $\{r_{0},\lambda_{2},\ldots, \lambda_{p_{0}}\}$, $\{r_{1}, \beta_{2},$ $ \beta_{3},\ldots, \beta_{p_{1}}\}$ and  $\{r_{2}, \gamma_{2}, \gamma_{3}, \ldots, \gamma_{p_{2}}\}$ are the $A-$ eigenvalues of $G_{0}$, $G_{1}$ and $G_{2}$ respectively. Then $\epsilon-$ eigenvalues of $G_{0}\veebar (G_{1}\cup G_{2})$ consists of\begin{enumerate}
     \item $0$ with multiplicity $q_{0}-p_{0}.$
     \item $-2(1+\beta_{k}),$ $ k=2,\ldots,p_{1}.$
     \item $-2(1+\gamma_{l}),$ $ l=2,\ldots,p_{2}.$
     \item $-2(1+\lambda_{i})\pm \sqrt{4(1+\lambda_{i})^{2}+9(\lambda_{i}+r_{0})},$ $ i=2,\ldots,p_{0}.$
     \item $4$ eigenvalues of the equitable quotient matrix of $\epsilon(G_{0}\veebar (G_{1}\cup G_{2}))$, \begin{center} 
     $\begin{pmatrix}
         4(p_{0}-1-r_{0})&3(q_{0}-r_{0})&2p_{1}&2p_{2}\\
         3(p_{0}-2)&0&0&0\\
         2p_{0}&0&2(p_{1}-1-r_{1})&2p_{2}\\
         2p_{0}&0&2p_{1}&2(p_{2}-1-r_{2})
     \end{pmatrix}.$
     \end{center}
     \end{enumerate}
\end{theorem}
\begin{proof}
    Proof is similar to Theorem\ref{subdivisionvertexunion}.
\end{proof}
\begin{remark}
    For any three connected graphs $G_{0}$, $G_{1}$ and $G_{2}$ the matrix $\epsilon(G_{0}\veebar (G_{1}\cup G_{2}))$ is  irreducible. 
    This gives a new family of graphs having irreducible eccentricity matrix. 
\end{remark}
\begin{remark}
     For any connected graphs $G_{i}$,  $i=0, 1, \ldots, p$, the matrix  $\epsilon(G_{0}\veebar (\displaystyle\cup^{p}_{l=1} G_{l})$ is irreducible.
\end{remark}

\section{Applications}
This section will explore some of the applications of the results discussed in sections 3 and 4.

\subsection{ Eccentricity equienergetic graphs}
The construction of non cospectral equienergetic graphs is a significant problem in spectral graph theory. In \cite{mahato2023spectral},
Mahato et al. constructed a pair of non $\epsilon-$cospectral $\epsilon-$equienergetic graphs for every $p\geq 5.$  
Motivated by this, in this section we will discuss the construction some non $\epsilon-$cospectral $\epsilon-$equienergetic graphs.
\begin{theorem}
    For every $t\geq 3$, there exist a pair of $non$ $\epsilon-$ cospectral $\epsilon-$equienergetic graphs on $12t$ vertices. 
\end{theorem}
Proof. Let $G_{1}$ and $G_{2}$ be  two non-cospectral cubic graphs on $2t$  vertices as in Lemma \ref{cubicnoncospectral} with $spec (G_{1})=\{3,\lambda_{2},\ldots, \lambda_{2t}\}$ and $spec (G_{2})=\{3, \beta_{2},\ldots, \beta_{2t}\}$.
Using Lemma \ref{2linegraph},
\begin{equation*}
    spec L^{2}(G_{1})=\begin{pmatrix}
                    6 &\lambda_{2}+3 &\cdots& \lambda_{2t}+3& 0 &-2\\
                    1& 1 &\cdots& 1 &t &3t
                  \end{pmatrix}
\end{equation*}
and \begin{equation*}
     spec L^{2}(G_{2})=\begin{pmatrix}
                    6 &\beta_{2}+3 &\cdots& \beta_{2t}+3& 0 &-2\\
                    1& 1 &\cdots& 1 &t &3t
                  \end{pmatrix}.
\end{equation*}
Using  Theorem \ref{GjoinG} we have \begin{align*}
    spec_{\epsilon}(L^{2}(G_{1})\vee L^{2}(G_{1}) &=\begin{pmatrix}
        2(6t-7)& -2(\lambda_{i}+4) & -2 & 2\\\
        2&2&2t&6t
                                                   \end{pmatrix}, i=2, \ldots, 2t.\\
 spec_{\epsilon}(L^{2}(G_{2})\vee L^{2}(G_{2}) &=\begin{pmatrix}
        2(6t-7)& -2(\beta_{j}+4) & -2 & 2\\\
        2&2&2t&6t
                                                   \end{pmatrix}, j=2, \ldots, 2t.
                      \end{align*}
Then we have,
\begin{equation*}
    E_{\epsilon}(L^{2}(G_{1}) \vee L^{2}(G_{1}))=E_{\epsilon}(L^{2}(G_{2}) \vee L^{2}(G_{2}))=72t-56.
\end{equation*}

\begin{example}

Let $G_{1}$ and $G_{2}$ are two graphs as in Figure \ref{fig:1}.  
\begin{figure}[H]
 \centering
\begin{tikzpicture}[scale=.7]
\filldraw[fill=black](0,2)circle(0.1cm); 
\draw(0,2)--(1,1);
\draw(0,2)--(-1,1);
\draw(0,2)--(0,-2);

\filldraw[fill=black](1,1)circle(0.1cm);
\draw(1,1)--(1,-1);
\draw(1,1)--(-1,1);

\filldraw[fill=black](1,-1)circle(0.1cm);
\draw(1,-1)--(-1,-1);
\draw(1,-1)--(0,-2);

\filldraw[fill=black](0,-2)circle(0.1cm);
\draw(0,-2)--(-1,-1);

\filldraw[fill=black](-1,-1)circle(0.1cm);
\draw(-1,-1)--(-1,1);
\filldraw[fill=black](-1,1)circle(0.1cm);
\filldraw[fill=black](5,2)circle(0.1cm);
\draw(5,2)--(6,1);
\draw(5,2)--(4,1);
\draw(5,2)--(5,-2);

\filldraw[fill=black](6,1)circle(0.1cm);
\draw(6,1)--(6,-1);
\draw(6,1)--(4,-1);

\filldraw[fill=black](6,-1)circle(0.1cm);
\draw(6,-1)--(5,-2);
\draw(6,-1)--(4,1);
\filldraw[fill=black](5,-2)circle(0.1cm);
\filldraw[fill=black](4,-1)circle(0.1cm);
\draw(4,-1)--(4,1);
\draw(4,-1)--(5,-2);
\filldraw[fill=black](4,1)circle(0.1cm);

\end{tikzpicture}
 \caption{ $G_{1}$ and $G_{2}$ are non isomorphic non cospectral cubic graphs on 6 vertices.}
\label{fig:1}
\end{figure}
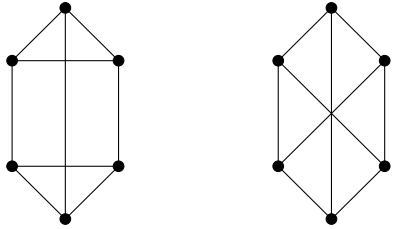

Then, $spec(G_{1})=\begin{pmatrix}
    3&1&0&-2\\
    1&1&2&2
\end{pmatrix}$ and $spec(G_{2})=\begin{pmatrix}
    3&0&3\\
    1&4&1
\end{pmatrix}.$ Now, $L^{2}(G_{l}) \vee L^{2}(G_{l}), l=1,2$ is a graph having 36 vertices. Also, its eccentricity spectrum is given by,
\begin{align*}
spec_{\epsilon}((L^{2}(G_{1}) \vee L^{2}(G_{1}))& =\begin{pmatrix}
    22&-10&-8&-4&-2&2\\
    2&2&4&4&6&18
\end{pmatrix}\\
spec_{\epsilon}((L^{2}(G_{2}) \vee L^{2}(G_{2}))& =\begin{pmatrix}
    22&-8&-2&2\\
    2&8&8&18
\end{pmatrix}.\\
\text { Then we have, }
                 E_{\epsilon}(L^{2}(G_{1}) \vee L^{2}(G_{1})) &= E_{\epsilon}(L^{2}(G_{2}) \vee L^{2}(G_{2}))=160.
\end{align*}
    
\end{example}

\begin{figure}[H]
 \centering
\begin{tikzpicture}[scale=.55]

\filldraw[fill=black](4,0)circle(0.1cm); 
\draw(4,0)--(-4,0);
\draw(4,0)--(-0.695,3.939);
\draw(4,0)--(3.064,2.571);
\draw(4,0)--(-3.064,2.571);
\draw(4,0)--(3.759,-1.368);
\draw(4,0)--(3.759,1.368);
\filldraw[fill=black](-4,0)circle(0.1cm); 
\draw(-4,0)--(-0.695,3.939);
\draw(-4,0)--(-3.064,-2.571);
\draw(-4,0)--(-3.064,2.571);
\draw(-4,0)--(-3.759,1.368);
\draw(-4,0)--(-3.759,-1.368);
\filldraw[fill=black](3.759,1.368)circle(0.1cm);
\draw(3.759,1.368)--(-2,-3.464);
\draw(3.759,1.368)--(3.064,-2.571);
\draw(3.759,1.368)--(3.064,2.571);
\draw(3.759,1.368)--(3.759,-1.368);
\draw(3.759,1.368)--(-3.759,-1.368);
\filldraw[fill=black](-3.759,1.368)circle(0.1cm);
\draw(-3.759,1.368)--(-2,3.464);
\draw(-3.759,1.368)--(2,3.464);
\draw(-3.759,1.368)--(-3.064,-2.571);
\draw(-3.759,1.368)--(-3.064,2.571);
\draw(-3.759,1.368)--(-3.759,-1.368);
\filldraw[fill=black](3.759,-1.368)circle(0.1cm);
\draw(3.759,-1.368)--(0.695,-3.939);
\draw(3.759,-1.368)--(2,-3.464);
\draw(3.759,-1.368)--(3.064,-2.571);
\draw(3.759,-1.368)--(3.064,2.571);
\filldraw[fill=black](-3.759,-1.368)circle(0.1cm);
\draw(-3.759,-1.368)--(-2,-3.464);
\draw(-3.759,-1.368)--(3.064,-2.571);
\draw(-3.759,-1.368)--(-3.064,-2.571);
\filldraw[fill=black](3.064,2.571)circle(0.1cm);
\draw(3.064,2.571)--(2,3.464);
\draw(3.064,2.571)--(0.695,3.939);
\draw(3.064,2.571)--(-0.695,3.939);
\filldraw[fill=black](-3.064,2.571)circle(0.1cm);
\draw(-3.064,2.571)--(-2,3.464);
\draw(-3.064,2.571)--(-0.695,3.939);
\draw(-3.064,2.571)--(2,3.464);
\filldraw[fill=black](3.064,-2.571)circle(0.1cm);
\draw(3.064,-2.571)--(2,-3.464);
\draw(3.064,-2.571)--(0.695,-3.939);
\draw(3.064,-2.571)--(-2,-3.464);
\filldraw[fill=black](-3.064,-2.571)circle(0.1cm);
\draw(-3.064,-2.571)--(-2,-3.464);
\draw(-3.064,-2.571)--(-0.695,-3.939);
\draw(-3.064,-2.571)--(0.695,-3.939);
\filldraw[fill=black](2,3.464)circle(0.1cm);
\draw(2,3.464)--(-2,3.464);
\draw(2,3.464)--(-0.695,3.939);
\draw(2,3.464)--(0.695,3.939);
\filldraw[fill=black](-2,3.464)circle(0.1cm);
\draw(-2,3.464)--(2,-3.464);
\draw(-2,3.464)--(-0.695,-3.939);
\draw(-2,3.464)--(0.695,3.939);
\filldraw[fill=black](2,-3.464)circle(0.1cm);
\draw(2,-3.464)--(0.695,-3.939);
\draw(2,-3.464)--(-0.695,-3.939);
\draw(2,-3.464)--(0.695,3.939);
\filldraw[fill=black](-2,-3.464)circle(0.1cm);
\draw(-2,-3.464)--(-0.695,-3.939);
\draw(-2,-3.464)--(0.695,-3.939);
\filldraw[fill=black](0.695,3.939)circle(0.1cm);
\draw(0.695,3.939)--(-0.695,3.939);
\draw(0.695,3.939)--(-0.695,-3.939);
\filldraw[fill=black](-0.695,3.939)circle(0.1cm);
\filldraw[fill=black](0.695,-3.939)circle(0.1cm);
\draw(0.695,-3.939)--(-0.695,-3.939);
\filldraw[fill=black](-0.695,-3.939)circle(0.1cm);
\filldraw[fill=black](16,0)circle(0.1cm); 
\draw(16,0)--(8,0);
\draw(16,0)--(11.305,3.939);
\draw(16,0)--(15.064,2.571);
\draw(16,0)--(8.936,2.571);
\draw(16,0)--(15.759,-1.368);
\draw(16,0)--(15.759,1.368);
\filldraw[fill=black](8,0)circle(0.1cm); 
\draw(8,0)--(11.305,3.939);
\draw(8,0)--(8.936,-2.571);
\draw(8,0)--(8.936,2.571);
\draw(8,0)--(8.241,1.368);
\draw(8,0)--(8.241,-1.368);
\filldraw[fill=black](15.759,1.368)circle(0.1cm);
\draw(15.759,1.368)--(10,-3.464);
\draw(15.759,1.368)--(15.064,-2.571);
\draw(15.759,1.368)--(15.064,2.571);
\draw(15.759,1.368)--(15.759,-1.368);
\draw(15.759,1.368)--(8.241,-1.368);
\filldraw[fill=black](8.241,1.368)circle(0.1cm);
\draw(8.241,1.368)--(10,3.464);
\draw(8.241,1.368)--(14,3.464);
\draw(8.241,1.368)--(8.936,-2.571);
\draw(8.241,1.368)--(8.936,2.571);
\draw(8.241,1.368)--(8.241,-1.368);
\filldraw[fill=black](15.759,-1.368)circle(0.1cm);
\draw(15.759,-1.368)--(12.695,-3.939);
\draw(15.759,-1.368)--(14,-3.464);
\draw(15.759,-1.368)--(15.064,-2.571);
\draw(15.759,-1.368)--(15.064,2.571);
\filldraw[fill=black](8.241,-1.368)circle(0.1cm);
\draw(8.241,-1.368)--(10,-3.464);
\draw(8.241,-1.368)--(15.064,-2.571);
\draw(8.241,-1.368)--(8.936,-2.571);
\filldraw[fill=black](15.064,2.571)circle(0.1cm);
\draw(15.064,2.571)--(14,3.464);
\draw(15.064,2.571)--(12.695,3.939);
\draw(15.064,2.571)--(11.305,3.939);
\filldraw[fill=black](8.936,2.571)circle(0.1cm);
\draw(8.936,2.571)--(10,3.464);
\draw(8.936,2.571)--(11.305,3.939);
\draw(8.936,2.571)--(14,3.464);
\filldraw[fill=black](15.064,-2.571)circle(0.1cm);
\draw(15.064,-2.571)--(14,-3.464);
\draw(15.064,-2.571)--(12.695,-3.939);
\draw(15.064,-2.571)--(10,-3.464);
\filldraw[fill=black](8.936,-2.571)circle(0.1cm);
\draw(8.936,-2.571)--(10,-3.464);
\draw(8.936,-2.571)--(11.305,-3.939);
\draw(8.936,-2.571)--(12.695,-3.939);
\filldraw[fill=black](14,3.464)circle(0.1cm);
\draw(14,3.464)--(10,3.464);
\draw(14,3.464)--(11.305,3.939);
\draw(14,3.464)--(12.695,3.939);
\filldraw[fill=black](10,3.464)circle(0.1cm);
\draw(10,3.464)--(14,-3.464);
\draw(10,3.464)--(11.305,-3.939);
\draw(10,3.464)--(12.695,3.939);
\filldraw[fill=black](14,-3.464)circle(0.1cm);
\draw(14,-3.464)--(12.695,-3.939);
\draw(14,-3.464)--(11.305,-3.939);
\draw(14,-3.464)--(12.695,3.939);
\filldraw[fill=black](10,-3.464)circle(0.1cm);
\draw(10,-3.464)--(11.305,-3.939);
\draw(10,-3.464)--(12.695,-3.939);
 \filldraw[fill=black](12.695,3.939)circle(0.1cm);
\draw(12.695,3.939)--(11.305,3.939);
\draw(12.695,3.939)--(11.305,-3.939);
\filldraw[fill=black](11.305,3.939)circle(0.1cm);
\filldraw[fill=black](12.695,-3.939)circle(0.1cm);
\draw(12.695,-3.939)--(11.305,-3.939);
\filldraw[fill=black](11.305,-3.939)circle(0.1cm);

\draw(4,0)--(16,0);
\draw(4,0)--(8,0);
\draw(4,0)--(15.759,1.368);
\draw(4,0)--(8.241,1.368);
\draw(4,0)--(15.759,-1.368);
\draw(4,0)--(8.241,-1.368);
\draw(4,0)--(15.064,2.571);
\draw(4,0)--(8.936,2.571);
\draw(4,0)--(15.064,-2.571);
\draw(4,0)--(8.936,-2.571);
\draw(4,0)--(14,3.464);
\draw(4,0)--(10,3.464);
\draw(4,0)--(14,-3.464);
\draw(4,0)--(10,-3.464);
\draw(4,0)--(12.695,3.939);
\draw(4,0)--(11.305,3.939);
\draw(4,0)--(12.695,-3.939);
\draw(4,0)--(11.305,-3.939);
\draw(-4,0)--(16,0);
\draw(-4,0)--(8,0);
\draw(-4,0)--(15.759,1.368);
\draw(-4,0)--(8.241,1.368);
\draw(-4,0)--(15.759,-1.368);
\draw(-4,0)--(8.241,-1.368);
\draw(-4,0)--(15.064,2.571);
\draw(-4,0)--(8.936,2.571);
\draw(-4,0)--(15.064,-2.571);
\draw(-4,0)--(8.936,-2.571);
\draw(-4,0)--(14,3.464);
\draw(-4,0)--(10,3.464);
\draw(-4,0)--(14,-3.464);
\draw(-4,0)--(10,-3.464);
\draw(-4,0)--(12.695,3.939);
\draw(-4,0)--(11.305,3.939);
\draw(-4,0)--(12.695,-3.939);
\draw(-4,0)--(11.305,-3.939);

\draw(3.759,1.368)--(16,0);
\draw(3.759,1.368)--(8,0);
\draw(3.759,1.368)--(15.759,1.368);
\draw(3.759,1.368)--(8.241,1.368);
\draw(3.759,1.368)--(15.759,-1.368);
\draw(3.759,1.368)--(8.241,-1.368);
\draw(3.759,1.368)--(15.064,2.571);
\draw(3.759,1.368)--(8.936,2.571);
\draw(3.759,1.368)--(15.064,-2.571);
\draw(3.759,1.368)--(8.936,-2.571);
\draw(3.759,1.368)--(14,3.464);
\draw(3.759,1.368)--(10,3.464);
\draw(3.759,1.368)--(14,-3.464);
\draw(3.759,1.368)--(10,-3.464);
\draw(3.759,1.368)--(12.695,3.939);
\draw(3.759,1.368)--(11.305,3.939);
\draw(3.759,1.368)--(12.695,-3.939);
\draw(3.759,1.368)--(11.305,-3.939);

\draw(-3.759,1.368)--(16,0);
\draw(-3.759,1.368)--(8,0);
\draw(-3.759,1.368)--(15.759,1.368);
\draw(-3.759,1.368)--(8.241,1.368);
\draw(-3.759,1.368)--(15.759,-1.368);
\draw(-3.759,1.368)--(8.241,-1.368);
\draw(-3.759,1.368)--(15.064,2.571);
\draw(-3.759,1.368)--(8.936,2.571);
\draw(-3.759,1.368)--(15.064,-2.571);
\draw(-3.759,1.368)--(8.936,-2.571);
\draw(-3.759,1.368)--(14,3.464);
\draw(-3.759,1.368)--(10,3.464);
\draw(-3.759,1.368)--(14,-3.464);
\draw(-3.759,1.368)--(10,-3.464);
\draw(-3.759,1.368)--(12.695,3.939);
\draw(-3.759,1.368)--(11.305,3.939);
\draw(-3.759,1.368)--(12.695,-3.939);
\draw(-3.759,1.368)--(11.305,-3.939);

\draw(3.759,-1.368)--(16,0);
\draw(3.759,-1.368)--(8,0);
\draw(3.759,-1.368)--(15.759,1.368);
\draw(3.759,-1.368)--(8.241,1.368);
\draw(3.759,-1.368)--(15.759,-1.368);
\draw(3.759,-1.368)--(8.241,-1.368);
\draw(3.759,-1.368)--(15.064,2.571);
\draw(3.759,-1.368)--(8.936,2.571);
\draw(3.759,-1.368)--(15.064,-2.571);
\draw(3.759,-1.368)--(8.936,-2.571);
\draw(3.759,-1.368)--(14,3.464);
\draw(3.759,-1.368)--(10,3.464);
\draw(3.759,-1.368)--(14,-3.464);
\draw(3.759,-1.368)--(10,-3.464);
\draw(3.759,-1.368)--(12.695,3.939);
\draw(3.759,-1.368)--(11.305,3.939);
\draw(3.759,-1.368)--(12.695,-3.939);
\draw(3.759,-1.368)--(11.305,-3.939);

\draw(-3.759,-1.368)--(16,0);
\draw(-3.759,-1.368)--(8,0);
\draw(-3.759,-1.368)--(15.759,1.368);
\draw(-3.759,-1.368)--(8.241,1.368);
\draw(-3.759,-1.368)--(15.759,-1.368);
\draw(-3.759,-1.368)--(8.241,-1.368);
\draw(-3.759,-1.368)--(15.064,2.571);
\draw(-3.759,-1.368)--(8.936,2.571);
\draw(-3.759,-1.368)--(15.064,-2.571);
\draw(-3.759,-1.368)--(8.936,-2.571);
\draw(-3.759,-1.368)--(14,3.464);
\draw(-3.759,-1.368)--(10,3.464);
\draw(-3.759,-1.368)--(14,-3.464);
\draw(-3.759,-1.368)--(10,-3.464);
\draw(-3.759,-1.368)--(12.695,3.939);
\draw(-3.759,-1.368)--(11.305,3.939);
\draw(-3.759,-1.368)--(12.695,-3.939);
\draw(-3.759,-1.368)--(11.305,-3.939);

\draw(3.064,2.571)--(16,0);
\draw(3.064,2.571)--(8,0);
\draw(3.064,2.571)--(15.759,1.368);
\draw(3.064,2.571)--(8.241,1.368);
\draw(3.064,2.571)--(15.759,-1.368);
\draw(3.064,2.571)--(8.241,-1.368);
\draw(3.064,2.571)--(15.064,2.571);
\draw(3.064,2.571)--(8.936,2.571);
\draw(3.064,2.571)--(15.064,-2.571);
\draw(3.064,2.571)--(8.936,-2.571);
\draw(3.064,2.571)--(14,3.464);
\draw(3.064,2.571)--(10,3.464);
\draw(3.064,2.571)--(14,-3.464);
\draw(3.064,2.571)--(10,-3.464);
\draw(3.064,2.571)--(12.695,3.939);
\draw(3.064,2.571)--(11.305,3.939);
\draw(3.064,2.571)--(12.695,-3.939);
\draw(3.064,2.571)--(11.305,-3.939);

\draw(-3.064,2.571)--(16,0);
\draw(-3.064,2.571)--(8,0);
\draw(-3.064,2.571)--(15.759,1.368);
\draw(-3.064,2.571)--(8.241,1.368);
\draw(-3.064,2.571)--(15.759,-1.368);
\draw(-3.064,2.571)--(8.241,-1.368);
\draw(-3.064,2.571)--(15.064,2.571);
\draw(-3.064,2.571)--(8.936,2.571);
\draw(-3.064,2.571)--(15.064,-2.571);
\draw(-3.064,2.571)--(8.936,-2.571);
\draw(-3.064,2.571)--(14,3.464);
\draw(-3.064,2.571)--(10,3.464);
\draw(-3.064,2.571)--(14,-3.464);
\draw(-3.064,2.571)--(10,-3.464);
\draw(-3.064,2.571)--(12.695,3.939);
\draw(-3.064,2.571)--(11.305,3.939);
\draw(-3.064,2.571)--(12.695,-3.939);
\draw(-3.064,2.571)--(11.305,-3.939);

\draw(3.064,-2.571)--(16,0);
\draw(3.064,-2.571)--(8,0);
\draw(3.064,-2.571)--(15.759,1.368);
\draw(3.064,-2.571)--(8.241,1.368);
\draw(3.064,-2.571)--(15.759,-1.368);
\draw(3.064,-2.571)--(8.241,-1.368);
\draw(3.064,-2.571)--(15.064,2.571);
\draw(3.064,-2.571)--(8.936,2.571);
\draw(3.064,-2.571)--(15.064,-2.571);
\draw(3.064,-2.571)--(8.936,-2.571);
\draw(3.064,-2.571)--(14,3.464);
\draw(3.064,-2.571)--(10,3.464);
\draw(3.064,-2.571)--(14,-3.464);
\draw(3.064,-2.571)--(10,-3.464);
\draw(3.064,-2.571)--(12.695,3.939);
\draw(3.064,-2.571)--(11.305,3.939);
\draw(3.064,-2.571)--(12.695,-3.939);
\draw(3.064,-2.571)--(11.305,-3.939);

\draw(-3.064,-2.571)--(16,0);
\draw(-3.064,-2.571)--(8,0);
\draw(-3.064,-2.571)--(15.759,1.368);
\draw(-3.064,-2.571)--(8.241,1.368);
\draw(-3.064,-2.571)--(15.759,-1.368);
\draw(-3.064,-2.571)--(8.241,-1.368);
\draw(-3.064,-2.571)--(15.064,2.571);
\draw(-3.064,-2.571)--(8.936,2.571);
\draw(-3.064,-2.571)--(15.064,-2.571);
\draw(-3.064,-2.571)--(8.936,-2.571);
\draw(-3.064,-2.571)--(14,3.464);
\draw(-3.064,-2.571)--(10,3.464);
\draw(-3.064,-2.571)--(14,-3.464);
\draw(-3.064,-2.571)--(10,-3.464);
\draw(-3.064,-2.571)--(12.695,3.939);
\draw(-3.064,-2.571)--(11.305,3.939);
\draw(-3.064,-2.571)--(12.695,-3.939);
\draw(-3.064,-2.571)--(11.305,-3.939);

\draw(2,3.464)--(16,0);
\draw(2,3.464)--(8,0);
\draw(2,3.464)--(15.759,1.368);
\draw(2,3.464)--(8.241,1.368);
\draw(2,3.464)--(15.759,-1.368);
\draw(2,3.464)--(8.241,-1.368);
\draw(2,3.464)--(15.064,2.571);
\draw(2,3.464)--(8.936,2.571);
\draw(2,3.464)--(15.064,-2.571);
\draw(2,3.464)--(8.936,-2.571);
\draw(2,3.464)--(14,3.464);
\draw(2,3.464)--(10,3.464);
\draw(2,3.464)--(14,-3.464);
\draw(2,3.464)--(10,-3.464);
\draw(2,3.464)--(12.695,3.939);
\draw(2,3.464)--(11.305,3.939);
\draw(2,3.464)--(12.695,-3.939);
\draw(2,3.464)--(11.305,-3.939);

\draw(-2,3.464)--(16,0);
\draw(-2,3.464)--(8,0);
\draw(-2,3.464)--(15.759,1.368);
\draw(-2,3.464)--(8.241,1.368);
\draw(-2,3.464)--(15.759,-1.368);
\draw(-2,3.464)--(8.241,-1.368);
\draw(-2,3.464)--(15.064,2.571);
\draw(-2,3.464)--(8.936,2.571);
\draw(-2,3.464)--(15.064,-2.571);
\draw(-2,3.464)--(8.936,-2.571);
\draw(-2,3.464)--(14,3.464);
\draw(-2,3.464)--(10,3.464);
\draw(-2,3.464)--(14,-3.464);
\draw(-2,3.464)--(10,-3.464);
\draw(-2,3.464)--(12.695,3.939);
\draw(-2,3.464)--(11.305,3.939);
\draw(-2,3.464)--(12.695,-3.939);
\draw(-2,3.464)--(11.305,-3.939);

\draw(2,-3.464)--(16,0);
\draw(2,-3.464)--(8,0);
\draw(2,-3.464)--(15.759,1.368);
\draw(2,-3.464)--(8.241,1.368);
\draw(2,-3.464)--(15.759,-1.368);
\draw(2,-3.464)--(8.241,-1.368);
\draw(2,-3.464)--(15.064,2.571);
\draw(2,-3.464)--(8.936,2.571);
\draw(2,-3.464)--(15.064,-2.571);
\draw(2,-3.464)--(8.936,-2.571);
\draw(2,-3.464)--(14,3.464);
\draw(2,-3.464)--(10,3.464);
\draw(2,-3.464)--(14,-3.464);
\draw(2,-3.464)--(10,-3.464);
\draw(2,-3.464)--(12.695,3.939);
\draw(2,-3.464)--(11.305,3.939);
\draw(2,-3.464)--(12.695,-3.939);
\draw(2,-3.464)--(11.305,-3.939);

\draw(-2,-3.464)--(16,0);
\draw(-2,-3.464)--(8,0);
\draw(-2,-3.464)--(15.759,1.368);
\draw(-2,-3.464)--(8.241,1.368);
\draw(-2,-3.464)--(15.759,-1.368);
\draw(-2,-3.464)--(8.241,-1.368);
\draw(-2,-3.464)--(15.064,2.571);
\draw(-2,-3.464)--(8.936,2.571);
\draw(-2,-3.464)--(15.064,-2.571);
\draw(-2,-3.464)--(8.936,-2.571);
\draw(-2,-3.464)--(14,3.464);
\draw(-2,-3.464)--(10,3.464);
\draw(-2,-3.464)--(14,-3.464);
\draw(-2,-3.464)--(10,-3.464);
\draw(-2,-3.464)--(12.695,3.939);
\draw(-2,-3.464)--(11.305,3.939);
\draw(-2,-3.464)--(12.695,-3.939);
\draw(-2,-3.464)--(11.305,-3.939);

\draw(0.695,3.939)--(16,0);
\draw(0.695,3.939)--(8,0);
\draw(0.695,3.939)--(15.759,1.368);
\draw(0.695,3.939)--(8.241,1.368);
\draw(0.695,3.939)--(15.759,-1.368);
\draw(0.695,3.939)--(8.241,-1.368);
\draw(0.695,3.939)--(15.064,2.571);
\draw(0.695,3.939)--(8.936,2.571);
\draw(0.695,3.939)--(15.064,-2.571);
\draw(0.695,3.939)--(8.936,-2.571);
\draw(0.695,3.939)--(14,3.464);
\draw(0.695,3.939)--(10,3.464);
\draw(0.695,3.939)--(14,-3.464);
\draw(0.695,3.939)--(10,-3.464);
\draw(0.695,3.939)--(12.695,3.939);
\draw(0.695,3.939)--(11.305,3.939);
\draw(0.695,3.939)--(12.695,-3.939);
\draw(0.695,3.939)--(11.305,-3.939);

\draw(-0.695,3.939)--(16,0);
\draw(-0.695,3.939)--(8,0);
\draw(-0.695,3.939)--(15.759,1.368);
\draw(-0.695,3.939)--(8.241,1.368);
\draw(-0.695,3.939)--(15.759,-1.368);
\draw(-0.695,3.939)--(8.241,-1.368);
\draw(-0.695,3.939)--(15.064,2.571);
\draw(-0.695,3.939)--(8.936,2.571);
\draw(-0.695,3.939)--(15.064,-2.571);
\draw(-0.695,3.939)--(8.936,-2.571);
\draw(-0.695,3.939)--(14,3.464);
\draw(-0.695,3.939)--(10,3.464);
\draw(-0.695,3.939)--(14,-3.464);
\draw(-0.695,3.939)--(10,-3.464);
\draw(-0.695,3.939)--(12.695,3.939);
\draw(-0.695,3.939)--(11.305,3.939);
\draw(-0.695,3.939)--(12.695,-3.939);
\draw(-0.695,3.939)--(11.305,-3.939);

\draw(0.695,-3.939)--(16,0);
\draw(0.695,-3.939)--(8,0);
\draw(0.695,-3.939)--(15.759,1.368);
\draw(0.695,-3.939)--(8.241,1.368);
\draw(0.695,-3.939)--(15.759,-1.368);
\draw(0.695,-3.939)--(8.241,-1.368);
\draw(0.695,-3.939)--(15.064,2.571);
\draw(0.695,-3.939)--(8.936,2.571);
\draw(0.695,-3.939)--(15.064,-2.571);
\draw(0.695,-3.939)--(8.936,-2.571);
\draw(0.695,-3.939)--(14,3.464);
\draw(0.695,-3.939)--(10,3.464);
\draw(0.695,-3.939)--(14,-3.464);
\draw(0.695,-3.939)--(10,-3.464);
\draw(0.695,-3.939)--(12.695,3.939);
\draw(0.695,-3.939)--(11.305,3.939);
\draw(0.695,-3.939)--(12.695,-3.939);
\draw(0.695,-3.939)--(11.305,-3.939);

\draw(-0.695,-3.939)--(16,0);
\draw(-0.695,-3.939)--(8,0);
\draw(-0.695,-3.939)--(15.759,1.368);
\draw(-0.695,-3.939)--(8.241,1.368);
\draw(-0.695,-3.939)--(15.759,-1.368);
\draw(-0.695,-3.939)--(8.241,-1.368);
\draw(-0.695,-3.939)--(15.064,2.571);
\draw(-0.695,-3.939)--(8.936,2.571);
\draw(-0.695,-3.939)--(15.064,-2.571);
\draw(-0.695,-3.939)--(8.936,-2.571);
\draw(-0.695,-3.939)--(14,3.464);
\draw(-0.695,-3.939)--(10,3.464);
\draw(-0.695,-3.939)--(14,-3.464);
\draw(-0.695,-3.939)--(10,-3.464);
\draw(-0.695,-3.939)--(12.695,3.939);
\draw(-0.695,-3.939)--(11.305,3.939);
\draw(-0.695,-3.939)--(12.695,-3.939);
\draw(-0.695,-3.939)--(11.305,-3.939);
\filldraw[fill=black](4,-9)circle(0.1cm); 
\draw(4,-9)--(3.759,-7.632);
\draw(4,-9)--(3.759,-10.368);
\draw(4,-9)--(3.064,-11.571);
\draw(4,-9)--(2,-12.464);
\draw(4,-9)--(0.695,-12.939);
\draw(4,-9)--(-3.064,-6.429);

\filldraw[fill=black](-4,-9)circle(0.1cm); 
\draw(-4,-9)--(-3.759,-7.632);
\draw(-4,-9)--(-2,-5.536);
\draw(-4,-9)--(2,-5.536);
\draw(-4,-9)--(3.064,-6.429);
\draw(-4,-9)--(-3.759,-10.368);
\draw(-4,-9)--(2,-12.464);

\filldraw[fill=black](3.759,-7.632)circle(0.1cm);
\draw(3.759,-7.632)--(3.064,-6.429);
\draw(3.759,-7.632)--(0.695,-12.939);
\draw(3.759,-7.632)--(-0.695,-12.939);
\draw(3.759,-7.632)--(-3.759,-10.368);
\draw(3.759,-7.632)--(-3.064,-6.429);

\filldraw[fill=black](-3.759,-7.632)circle(0.1cm);
\draw(-3.759,-7.632)--(-3.759,-10.368);
\draw(-3.759,-7.632)--(-3.064,-11.571);
\draw(-3.759,-7.632)--(2,-12.464);
\draw(-3.759,-7.632)--(3.759,-10.368);
\draw(-3.759,-7.632)--(0.695,-5.061);

\filldraw[fill=black](3.759,-10.368)circle(0.1cm);
\draw(3.759,-10.368)--(3.064,-11.571);
\draw(3.759,-10.368)--(2,-12.464);
\draw(3.759,-10.368)--(-3.064,-11.571);
\draw(3.759,-10.368)--(0.695,-5.061);

\filldraw[fill=black](-3.759,-10.368)circle(0.1cm);
\draw(-3.759,-10.368)--(3.064,-6.429);
\draw(-3.759,-10.368)--(2,-12.464);
\draw(-3.759,-10.368)--(-0.695,-12.939);

\filldraw[fill=black](3.064,-6.429)circle(0.1cm);
\draw(3.064,-6.429)--(2,-5.536);
\draw(3.064,-6.429)--(-2,-5.536);
\draw(3.064,-6.429)--(-0.695,-12.939);

\filldraw[fill=black](-3.064,-6.429)circle(0.1cm);
\draw(-3.064,-6.429)--(-2,-5.536);
\draw(-3.064,-6.429)--(-0.695,-5.061);
\draw(-3.064,-6.429)--(3.064,-11.571);
\draw(-3.064,-6.429)--(0.695,-12.939);

\filldraw[fill=black](3.064,-11.571)circle(0.1cm);
\draw(3.064,-11.571)--(2,-12.464);
\draw(3.064,-11.571)--(-0.695,-5.061);
\draw(3.064,-11.571)--(-2,-5.536);

\filldraw[fill=black](-3.064,-11.571)circle(0.1cm);
\draw(-3.064,-11.571)--(-2,-12.464);
\draw(-3.064,-11.571)--(-0.695,-12.939);
\draw(-3.064,-11.571)--(0.695,-12.939);
\draw(-3.064,-11.571)--(0.695,-5.061);

\filldraw[fill=black](2,-5.536)circle(0.1cm);
\draw(2,-5.536)--(0.695,-5.061);
\draw(2,-5.536)--(-0.695,-5.061);
\draw(2,-5.536)--(-2,-5.536);
\draw(2,-5.536)--(-2,-12.464);

\filldraw[fill=black](-2,-5.536)circle(0.1cm);
\draw(-2,-5.536)--(-0.695,-5.061);

\filldraw[fill=black](2,-12.464)circle(0.1cm);

\filldraw[fill=black](-2,-12.464)circle(0.1cm);
\draw(-2,-12.464)--(-0.695,-12.939);
\draw(-2,-12.464)--(0.695,-12.939);
\draw(-2,-12.464)--(-0.695,-5.061);
\draw(-2,-12.464)--(0.695,-5.061);

\filldraw[fill=black](0.695,-5.061)circle(0.1cm);
\draw(0.695,-5.061)--(-0.695,-5.061);

\filldraw[fill=black](-0.695,-5.061)circle(0.1cm);

\filldraw[fill=black](0.695,-12.939)circle(0.1cm);
\draw(0.695,-12.939)--(-0.695,-12.939);

\filldraw[fill=black](-0.695,-12.939)circle(0.1cm);
\filldraw[fill=black](16,-9)circle(0.1cm); 
\draw(16,-9)--(15.759,-7.362);
\draw(16,-9)--(8.936,-6.429);
\draw(16,-9)--(15.759,-10.368);
\draw(16,-9)--(15.064,-11.571);
 \draw(16,-9)--(14,-12.464);
\draw(16,-9)--(12.965,-12.939);

\filldraw[fill=black](8,-9)circle(0.1cm); 
\draw(8,-9)--(8.241,-10.368);
\draw(8,-9)--(14,-12.464);
\draw(8,-9)--(8.241,-7.362);
\draw(8,-9)--(10,-5.536);
\draw(8,-9)--(14,-5.536);
\draw(8,-9)--(15.064,-6.429);

\filldraw[fill=black](15.759,-7.362)circle(0.1cm);
\draw(15.759,-7.362)--(15.064,-6.429);
\draw(15.759,-7.362)--(8.241,-10.368);
\draw(15.759,-7.362)--(8.936,-6.429);
\draw(15.759,-7.362)--(11.305,-12.939);
\draw(15.759,-7.362)--(12.965,-12.939);

\filldraw[fill=black](8.241,-7.362)circle(0.1cm);
\draw(8.241,-7.362)--(8.241,-10.368);
\draw(8.241,-7.362)--(8.936,-11.571);
\draw(8.241,-7.362)--(14,-12.464);
\draw(8.241,-7.362)--(15.759,-10.368);
\draw(8.241,-7.362)--(12.695,-5.061);

\filldraw[fill=black](15.759,-10.368)circle(0.1cm);
\draw(15.759,-10.368)--(12.695,-5.061);
\draw(15.759,-10.368)--(8.936,-11.571);
\draw(15.759,-10.368)--(14,-12.464);
\draw(15.759,-10.368)--(15.064,-11.571);

\filldraw[fill=black](8.241,-10.368)circle(0.1cm);
\draw(8.241,-10.368)--(11.305,-12.939);
\draw(8.241,-10.368)--(15.064,-6.429);
\draw(8.241,-10.368)--(14,-12.464);

\filldraw[fill=black](15.064,-6.429)circle(0.1cm);
\draw(15.064,-6.429)--(14,-5.536);
\draw(15.064,-6.429)--(10,-5.536);
\draw(15.064,-6.429)--(11.305,-12.939);

\filldraw[fill=black](8.936,-6.429)circle(0.1cm);
\draw(8.936,-6.429)--(10,-5.536);
\draw(8.936,-6.429)--(11.305,-5.061);
\draw(8.936,-6.429)--(15.064,-11.571);
\draw(8.936,-6.429)--(12.965,-12.939);

\filldraw[fill=black](15.064,-11.571)circle(0.1cm);
\draw(15.064,-11.571)--(14,-12.464);
\draw(15.064,-11.571)--(10,-5.536);
\draw(15.064,-11.571)--(11.305,-5.061);

\filldraw[fill=black](8.936,-11.571)circle(0.1cm);
\draw(8.936,-11.571)--(12.965,-12.939);
\draw(8.936,-11.571)--(11.305,-12.939);
\draw(8.936,-11.571)--(10,-12.464);
\draw(8.936,-11.571)--(12.695,-5.061);

\filldraw[fill=black](14,-5.536)circle(0.1cm);
\draw(14,-5.536)--(10,-5.536);
\draw(14,-5.536)--(11.305,-5.061);
\draw(14,-5.536)--(12.695,-5.061);
\draw(14,-5.536)--(10,-12.464);

\filldraw[fill=black](10,-5.536)circle(0.1cm);
\draw(10,-5.536)--(11.305,-5.061);

\filldraw[fill=black](14,-12.464)circle(0.1cm);

\filldraw[fill=black](10,-12.464)circle(0.1cm);
\draw(10,-12.464)--(11.305,-5.061);
\draw(10,-12.464)--(12.695,-5.061);
\draw(10,-12.464)--(12.965,-12.939);
\draw(10,-12.464)--(11.305,-12.939);

 \filldraw[fill=black](12.695,-5.061)circle(0.1cm);
\draw(12.695,-5.061)--(11.305,-5.061);

\filldraw[fill=black](11.305,-5.061)circle(0.1cm);

\filldraw[fill=black](12.965,-12.939)circle(0.1cm);
\draw(12.965,-12.939)--(11.305,-12.939);

\filldraw[fill=black](11.305,-12.939)circle(0.1cm);

\draw(4,-9)--(16,-9);
\draw(4,-9)--(8,-9);
\draw(4,-9)--(15.759,-7.362);
\draw(4,-9)--(8.241,-7.362);
\draw(4,-9)--(15.759,-10.368);
\draw(4,-9)--(8.241,-10.368);
\draw(4,-9)--(15.064,-6.429);
\draw(4,-9)--(8.936,-6.429);
\draw(4,-9)--(15.064,-11.571);
\draw(4,-9)--(8.936,-11.571);
\draw(4,-9)--(14,-5.536);
\draw(4,-9)--(10,-5.536);
\draw(4,-9)--(14,-12.464);
\draw(4,-9)--(10,-12.464);
\draw(4,-9)--(12.695,-5.061);
\draw(4,-9)--(11.305,-5.061);
\draw(4,-9)--(12.965,-12.939);
\draw(4,-9)--(11.305,-12.939);
\draw(-4,-9)--(16,-9);
\draw(-4,-9)--(8,-9);
\draw(-4,-9)--(15.759,-7.362);
\draw(-4,-9)--(8.241,-7.362);
\draw(-4,-9)--(15.759,-10.368);
\draw(-4,-9)--(8.241,-10.368);
\draw(-4,-9)--(15.064,-6.429);
\draw(-4,-9)--(8.936,-6.429);
\draw(-4,-9)--(15.064,-11.571);
\draw(-4,-9)--(8.936,-11.571);
\draw(-4,-9)--(14,-5.536);
\draw(-4,-9)--(10,-5.536);
\draw(-4,-9)--(14,-12.464);
\draw(-4,-9)--(10,-12.464);
\draw(-4,-9)--(12.695,-5.061);
\draw(-4,-9)--(11.305,-5.061);
\draw(-4,-9)--(12.965,-12.939);
\draw(-4,-9)--(11.305,-12.939);
\draw(3.759,-7.632)--(16,-9);
\draw(3.759,-7.632)--(8,-9);
\draw(3.759,-7.632)--(15.759,-7.362);
\draw(3.759,-7.632)--(8.241,-7.362);
\draw(3.759,-7.632)--(15.759,-10.368);
\draw(3.759,-7.632)--(8.241,-10.368);
\draw(3.759,-7.632)--(15.064,-6.429);
\draw(3.759,-7.632)--(8.936,-6.429);
\draw(3.759,-7.632)--(15.064,-11.571);
\draw(3.759,-7.632)--(8.936,-11.571);
\draw(3.759,-7.632)--(14,-5.536);
\draw(3.759,-7.632)--(10,-5.536);
\draw(3.759,-7.632)--(14,-12.464);
\draw(3.759,-7.632)--(10,-12.464);
\draw(3.759,-7.632)--(12.695,-5.061);
\draw(3.759,-7.632)--(11.305,-5.061);
\draw(3.759,-7.632)--(12.965,-12.939);
\draw(3.759,-7.632)--(11.305,-12.939);
\draw(-3.759,-7.632)--(16,-9);
\draw(-3.759,-7.632)--(8,-9);
\draw(-3.759,-7.632)--(15.759,-7.362);
\draw(-3.759,-7.632)--(8.241,-7.362);
\draw(-3.759,-7.632)--(15.759,-10.368);
\draw(-3.759,-7.632)--(8.241,-10.368);
\draw(-3.759,-7.632)--(15.064,-6.429);
\draw(-3.759,-7.632)--(8.936,-6.429);
\draw(-3.759,-7.632)--(15.064,-11.571);
\draw(-3.759,-7.632)--(8.936,-11.571);
\draw(-3.759,-7.632)--(14,-5.536);
\draw(-3.759,-7.632)--(10,-5.536);
\draw(-3.759,-7.632)--(14,-12.464);
\draw(-3.759,-7.632)--(10,-12.464);
\draw(-3.759,-7.632)--(12.695,-5.061);
\draw(-3.759,-7.632)--(11.305,-5.061);
\draw(-3.759,-7.632)--(12.965,-12.939);
\draw(-3.759,-7.632)--(11.305,-12.939);
\draw(3.759,-10.368)--(16,-9);
\draw(3.759,-10.368)--(8,-9);
\draw(3.759,-10.368)--(15.759,-7.362);
\draw(3.759,-10.368)--(8.241,-7.362);
\draw(3.759,-10.368)--(15.759,-10.368);
\draw(3.759,-10.368)--(8.241,-10.368);
\draw(3.759,-10.368)--(15.064,-6.429);
\draw(3.759,-10.368)--(8.936,-6.429);
\draw(3.759,-10.368)--(15.064,-11.571);
\draw(3.759,-10.368)--(8.936,-11.571);
\draw(3.759,-10.368)--(14,-5.536);
\draw(3.759,-10.368)--(10,-5.536);
\draw(3.759,-10.368)--(14,-12.464);
\draw(3.759,-10.368)--(10,-12.464);
\draw(3.759,-10.368)--(12.695,-5.061);
\draw(3.759,-10.368)--(11.305,-5.061);
\draw(3.759,-10.368)--(12.965,-12.939);
\draw(3.759,-10.368)--(11.305,-12.939);
\draw(-3.759,-10.368)--(16,-9);
\draw(-3.759,-10.368)--(8,-9);
\draw(-3.759,-10.368)--(15.759,-7.362);
\draw(-3.759,-10.368)--(8.241,-7.362);
\draw(-3.759,-10.368)--(15.759,-10.368);
\draw(-3.759,-10.368)--(8.241,-10.368);
\draw(-3.759,-10.368)--(15.064,-6.429);
\draw(-3.759,-10.368)--(8.936,-6.429);
\draw(-3.759,-10.368)--(15.064,-11.571);
\draw(-3.759,-10.368)--(8.936,-11.571);
\draw(-3.759,-10.368)--(14,-5.536);
\draw(-3.759,-10.368)--(10,-5.536);
\draw(-3.759,-10.368)--(14,-12.464);
\draw(-3.759,-10.368)--(10,-12.464);
\draw(-3.759,-10.368)--(12.695,-5.061);
\draw(-3.759,-10.368)--(11.305,-5.061);
\draw(-3.759,-10.368)--(12.965,-12.939);
\draw(-3.759,-10.368)--(11.305,-12.939);
\draw(3.064,-6.429)--(16,-9);
\draw(3.064,-6.429)--(8,-9);
\draw(3.064,-6.429)--(15.759,-7.362);
\draw(3.064,-6.429)--(8.241,-7.362);
\draw(3.064,-6.429)--(15.759,-10.368);
\draw(3.064,-6.429)--(8.241,-10.368);
\draw(3.064,-6.429)--(15.064,-6.429);
\draw(3.064,-6.429)--(8.936,-6.429);
\draw(3.064,-6.429)--(15.064,-11.571);
\draw(3.064,-6.429)--(8.936,-11.571);
\draw(3.064,-6.429)--(14,-5.536);
\draw(3.064,-6.429)--(10,-5.536);
\draw(3.064,-6.429)--(14,-12.464);
\draw(3.064,-6.429)--(10,-12.464);
\draw(3.064,-6.429)--(12.695,-5.061);
\draw(3.064,-6.429)--(11.305,-5.061);
\draw(3.064,-6.429)--(12.965,-12.939);
\draw(3.064,-6.429)--(11.305,-12.939);
\draw(-3.064,-6.429)--(16,-9);
\draw(-3.064,-6.429)--(8,-9);
\draw(-3.064,-6.429)--(15.759,-7.362);
\draw(-3.064,-6.429)--(8.241,-7.362);
\draw(-3.064,-6.429)--(15.759,-10.368);
\draw(-3.064,-6.429)--(8.241,-10.368);
\draw(-3.064,-6.429)--(15.064,-6.429);
\draw(-3.064,-6.429)--(8.936,-6.429);
\draw(-3.064,-6.429)--(15.064,-11.571);
\draw(-3.064,-6.429)--(8.936,-11.571);
\draw(-3.064,-6.429)--(14,-5.536);
\draw(-3.064,-6.429)--(10,-5.536);
\draw(-3.064,-6.429)--(14,-12.464);
\draw(-3.064,-6.429)--(10,-12.464);
\draw(-3.064,-6.429)--(12.695,-5.061);
\draw(-3.064,-6.429)--(11.305,-5.061);
\draw(-3.064,-6.429)--(12.965,-12.939);
\draw(-3.064,-6.429)--(11.305,-12.939);
\draw(3.064,-11.571)--(16,-9);
\draw(3.064,-11.571)--(8,-9);
\draw(3.064,-11.571)--(15.759,-7.362);
\draw(3.064,-11.571)--(8.241,-7.362);
\draw(3.064,-11.571)--(15.759,-10.368);
\draw(3.064,-11.571)--(8.241,-10.368);
\draw(3.064,-11.571)--(15.064,-6.429);
\draw(3.064,-11.571)--(8.936,-6.429);
\draw(3.064,-11.571)--(15.064,-11.571);
\draw(3.064,-11.571)--(8.936,-11.571);
\draw(3.064,-11.571)--(14,-5.536);
\draw(3.064,-11.571)--(10,-5.536);
\draw(3.064,-11.571)--(14,-12.464);
\draw(3.064,-11.571)--(10,-12.464);
\draw(3.064,-11.571)--(12.695,-5.061);
\draw(3.064,-11.571)--(11.305,-5.061);
\draw(3.064,-11.571)--(12.965,-12.939);
\draw(3.064,-11.571)--(11.305,-12.939);
\draw(-3.064,-11.571)--(16,-9);
\draw(-3.064,-11.571)--(8,-9);
\draw(-3.064,-11.571)--(15.759,-7.362);
\draw(-3.064,-11.571)--(8.241,-7.362);
\draw(-3.064,-11.571)--(15.759,-10.368);
\draw(-3.064,-11.571)--(8.241,-10.368);
\draw(-3.064,-11.571)--(15.064,-6.429);
\draw(-3.064,-11.571)--(8.936,-6.429);
\draw(-3.064,-11.571)--(15.064,-11.571);
\draw(-3.064,-11.571)--(8.936,-11.571);
\draw(-3.064,-11.571)--(14,-5.536);
\draw(-3.064,-11.571)--(10,-5.536);
\draw(-3.064,-11.571)--(14,-12.464);
\draw(-3.064,-11.571)--(10,-12.464);
\draw(-3.064,-11.571)--(12.695,-5.061);
\draw(-3.064,-11.571)--(11.305,-5.061);
\draw(-3.064,-11.571)--(12.965,-12.939);
\draw(-3.064,-11.571)--(11.305,-12.939);
\draw(2,-5.536)--(16,-9);
\draw(2,-5.536)--(8,-9);
\draw(2,-5.536)--(15.759,-7.362);
\draw(2,-5.536)--(8.241,-7.362);
\draw(2,-5.536)--(15.759,-10.368);
\draw(2,-5.536)--(8.241,-10.368);
\draw(2,-5.536)--(15.064,-6.429);
\draw(2,-5.536)--(8.936,-6.429);
\draw(2,-5.536)--(15.064,-11.571);
\draw(2,-5.536)--(8.936,-11.571);
\draw(2,-5.536)--(14,-5.536);
\draw(2,-5.536)--(10,-5.536);
\draw(2,-5.536)--(14,-12.464);
\draw(2,-5.536)--(10,-12.464);
\draw(2,-5.536)--(12.695,-5.061);
\draw(2,-5.536)--(11.305,-5.061);
\draw(2,-5.536)--(12.965,-12.939);
\draw(2,-5.536)--(11.305,-12.939);
\draw(-2,-5.536)--(16,-9);
\draw(-2,-5.536)--(8,-9);
\draw(-2,-5.536)--(15.759,-7.362);
\draw(-2,-5.536)--(8.241,-7.362);
\draw(-2,-5.536)--(15.759,-10.368);
\draw(-2,-5.536)--(8.241,-10.368);
\draw(-2,-5.536)--(15.064,-6.429);
\draw(-2,-5.536)--(8.936,-6.429);
\draw(-2,-5.536)--(15.064,-11.571);
\draw(-2,-5.536)--(8.936,-11.571);
\draw(-2,-5.536)--(14,-5.536);
\draw(-2,-5.536)--(10,-5.536);
\draw(-2,-5.536)--(14,-12.464);
\draw(-2,-5.536)--(10,-12.464);
\draw(-2,-5.536)--(12.695,-5.061);
\draw(-2,-5.536)--(11.305,-5.061);
\draw(-2,-5.536)--(12.965,-12.939);
\draw(-2,-5.536)--(11.305,-12.939);
\draw(2,-12.464)--(16,-9);
\draw(2,-12.464)--(8,-9);
\draw(2,-12.464)--(15.759,-7.362);
\draw(2,-12.464)--(8.241,-7.362);
\draw(2,-12.464)--(15.759,-10.368);
\draw(2,-12.464)--(8.241,-10.368);
\draw(2,-12.464)--(15.064,-6.429);
\draw(2,-12.464)--(8.936,-6.429);
\draw(2,-12.464)--(15.064,-11.571);
\draw(2,-12.464)--(8.936,-11.571);
\draw(2,-12.464)--(14,-5.536);
\draw(2,-12.464)--(10,-5.536);
\draw(2,-12.464)--(14,-12.464);
\draw(2,-12.464)--(10,-12.464);
\draw(2,-12.464)--(12.695,-5.061);
\draw(2,-12.464)--(11.305,-5.061);
\draw(2,-12.464)--(12.965,-12.939);
\draw(2,-12.464)--(11.305,-12.939);
\draw(-2,-12.464)--(16,-9);
\draw(-2,-12.464)--(8,-9);
\draw(-2,-12.464)--(15.759,-7.362);
\draw(-2,-12.464)--(8.241,-7.362);
\draw(-2,-12.464)--(15.759,-10.368);
\draw(-2,-12.464)--(8.241,-10.368);
\draw(-2,-12.464)--(15.064,-6.429);
\draw(-2,-12.464)--(8.936,-6.429);
\draw(-2,-12.464)--(15.064,-11.571);
\draw(-2,-12.464)--(8.936,-11.571);
\draw(-2,-12.464)--(14,-5.536);
\draw(-2,-12.464)--(10,-5.536);
\draw(-2,-12.464)--(14,-12.464);
\draw(-2,-12.464)--(10,-12.464);
\draw(-2,-12.464)--(12.695,-5.061);
\draw(-2,-12.464)--(11.305,-5.061);
\draw(-2,-12.464)--(12.965,-12.939);
\draw(-2,-12.464)--(11.305,-12.939);
\draw(0.695,-5.061)--(16,-9);
\draw(0.695,-5.061)--(8,-9);
\draw(0.695,-5.061)--(15.759,-7.362);
\draw(0.695,-5.061)--(8.241,-7.362);
\draw(0.695,-5.061)--(15.759,-10.368);
\draw(0.695,-5.061)--(8.241,-10.368);
\draw(0.695,-5.061)--(15.064,-6.429);
\draw(0.695,-5.061)--(8.936,-6.429);
\draw(0.695,-5.061)--(15.064,-11.571);
\draw(0.695,-5.061)--(8.936,-11.571);
\draw(0.695,-5.061)--(14,-5.536);
\draw(0.695,-5.061)--(10,-5.536);
\draw(0.695,-5.061)--(14,-12.464);
\draw(0.695,-5.061)--(10,-12.464);
\draw(0.695,-5.061)--(12.695,-5.061);
\draw(0.695,-5.061)--(11.305,-5.061);
\draw(0.695,-5.061)--(12.965,-12.939);
\draw(0.695,-5.061)--(11.305,-12.939);
\draw(-0.695,-5.061)--(16,-9);
\draw(-0.695,-5.061)--(8,-9);
\draw(-0.695,-5.061)--(15.759,-7.362);
\draw(-0.695,-5.061)--(8.241,-7.362);
\draw(-0.695,-5.061)--(15.759,-10.368);
\draw(-0.695,-5.061)--(8.241,-10.368);
\draw(-0.695,-5.061)--(15.064,-6.429);
\draw(-0.695,-5.061)--(8.936,-6.429);
\draw(-0.695,-5.061)--(15.064,-11.571);
\draw(-0.695,-5.061)--(8.936,-11.571);
\draw(-0.695,-5.061)--(14,-5.536);
\draw(-0.695,-5.061)--(10,-5.536);
\draw(-0.695,-5.061)--(14,-12.464);
\draw(-0.695,-5.061)--(10,-12.464);
\draw(-0.695,-5.061)--(12.695,-5.061);
\draw(-0.695,-5.061)--(11.305,-5.061);
\draw(-0.695,-5.061)--(12.965,-12.939);
\draw(-0.695,-5.061)--(11.305,-12.939);
\draw(0.695,-12.939)--(16,-9);
\draw(0.695,-12.939)--(8,-9);
\draw(0.695,-12.939)--(15.759,-7.362);
\draw(0.695,-12.939)--(8.241,-7.362);
\draw(0.695,-12.939)--(15.759,-10.368);
\draw(0.695,-12.939)--(8.241,-10.368);
\draw(0.695,-12.939)--(15.064,-6.429);
\draw(0.695,-12.939)--(8.936,-6.429);
\draw(0.695,-12.939)--(15.064,-11.571);
\draw(0.695,-12.939)--(8.936,-11.571);
\draw(0.695,-12.939)--(14,-5.536);
\draw(0.695,-12.939)--(10,-5.536);
\draw(0.695,-12.939)--(14,-12.464);
\draw(0.695,-12.939)--(10,-12.464);
\draw(0.695,-12.939)--(12.695,-5.061);
\draw(0.695,-12.939)--(11.305,-5.061);
\draw(0.695,-12.939)--(12.965,-12.939);
\draw(0.695,-12.939)--(11.305,-12.939);
\draw(-0.695,-12.939)--(16,-9);
\draw(-0.695,-12.939)--(8,-9);
\draw(-0.695,-12.939)--(15.759,-7.362);
\draw(-0.695,-12.939)--(8.241,-7.362);
\draw(-0.695,-12.939)--(15.759,-10.368);
\draw(-0.695,-12.939)--(8.241,-10.368);
\draw(-0.695,-12.939)--(15.064,-6.429);
\draw(-0.695,-12.939)--(8.936,-6.429);
\draw(-0.695,-12.939)--(15.064,-11.571);
\draw(-0.695,-12.939)--(8.936,-11.571);
\draw(-0.695,-12.939)--(14,-5.536);
\draw(-0.695,-12.939)--(10,-5.536);
\draw(-0.695,-12.939)--(14,-12.464);
\draw(-0.695,-12.939)--(10,-12.464);
\draw(-0.695,-12.939)--(12.695,-5.061);
\draw(-0.695,-12.939)--(11.305,-5.061);
\draw(-0.695,-12.939)--(12.965,-12.939);
\draw(-0.695,-12.939)--(11.305,-12.939);

 \end{tikzpicture}
 \caption{non isomorphic non $\epsilon-$cospectral $\epsilon$equienergetic graphs on 36-vertices with $E_{\epsilon}=160$}
\end{figure}
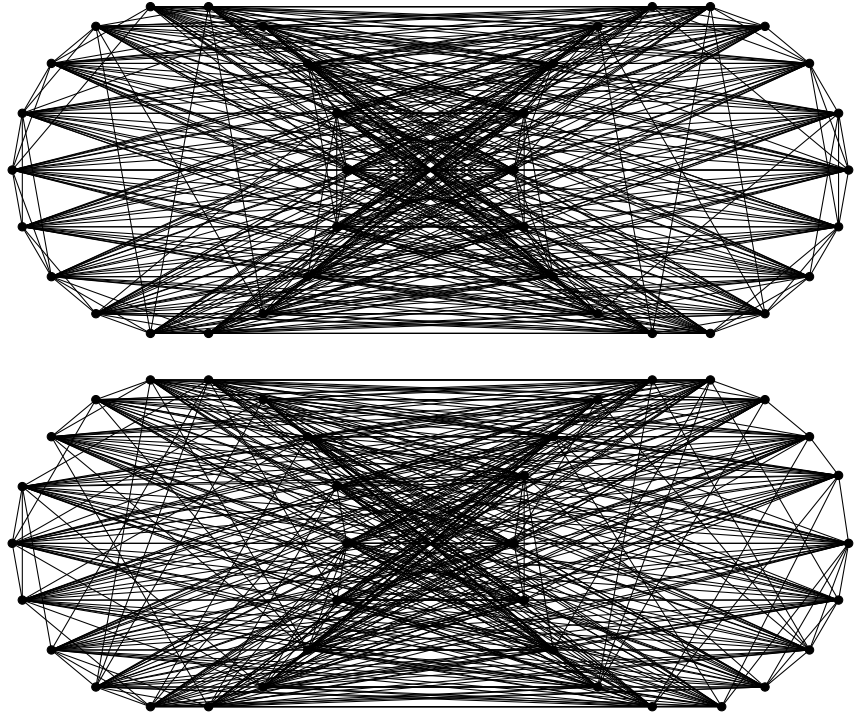

\begin{theorem}
    For every $t\geq3$, there exist a pair of $non$ $\epsilon-$cospectral $\epsilon-$equienergetic graphs on $6t+1$ vertices.
\end{theorem}
Proof. Let $G_{1}$ and $G_{2}$ are two non-cospectral cubic graphs on $2t$ vertices, as in Lemma \ref{cubicnoncospectral} with $spec (G_{1})=\{3,\lambda_{2},\ldots, \lambda_{2t}\}$ and $spec (G_{2})=\{3, \beta_{2},\ldots, \beta_{2t}\}$. Then using Lemma \ref{cubicnoncospectral} and Theorem \ref{GjoinK1},\begin{align*}
    spec_{\epsilon}(L^{2}(G_{1})\vee K_{1})&=\begin{pmatrix}
        (6t-7)\pm \sqrt{(6t-7)^{2}+6t}& -2(\lambda_{i}+4)&-2&2\\
        1& 1& t& 3t
    \end{pmatrix}, i=2,\ldots, 2t.\\
    spec_{\epsilon}(L^{2}(G_{2})\vee K_{1})&=\begin{pmatrix}
        (6t-7)\pm \sqrt{(6t-7)^{2}+6t}& -2(\beta_{j}+4)&-2&2\\
        1& 1& t& 3t
    \end{pmatrix}, j=2,\ldots, 2t.
\end{align*}
Then we have \begin{equation*}
    E_{\epsilon}(L^{2}(G_{1})\vee K_{1})=E_{\epsilon}(L^{2}(G_{2})\vee K_{1})=24t-14+2\sqrt{(6t-7)^{2}+6t}.
\end{equation*}
\begin{example}
Let $G_{1}$ and $G_{2}$ are two non cospectral cubic graphs as in figure $1$.
Now, $L^{2}(G_{l}) \vee K_{1}$, $l=1, 2$ is a graph having 19 vertices. Also, its eccentricity spectrum is given by, \begin{align*}
    spec_{\epsilon}(L^{2}(G_{1}) \vee K_{1})&=\begin{pmatrix}
        11\pm\sqrt{139}&-10&-8&-4&-2&2\\
        1&1&2&2&3&9
    \end{pmatrix}\\
    spec_{\epsilon}(L^{2}(G_{2}) \vee K_{1})&=\begin{pmatrix}
        11\pm\sqrt{139}&-8&-2&2\\
        1&4&4&9
    \end{pmatrix}.\\
 \text{ Then we have, }   
  E_{\epsilon}(L^{2}(G_{1}) \vee K_{1})&=E_{\epsilon}(L^{2}(G_{2})\vee K_{1})=81.579.
\end{align*}    
\end{example}

\begin{figure}[H]
 \centering
\begin{tikzpicture}[scale=.5]
\filldraw[fill=black](4,0)circle(0.1cm); 
\draw(4,0)--(-4,0);
\draw(4,0)--(-0.695,3.939);
\draw(4,0)--(3.064,2.571);
\draw(4,0)--(-3.064,2.571);
\draw(4,0)--(3.759,-1.368);
\draw(4,0)--(3.759,1.368);

\filldraw[fill=black](-4,0)circle(0.1cm); 
\draw(-4,0)--(-0.695,3.939);
\draw(-4,0)--(-3.064,-2.571);
\draw(-4,0)--(-3.064,2.571);
\draw(-4,0)--(-3.759,1.368);
\draw(-4,0)--(-3.759,-1.368);

\filldraw[fill=black](3.759,1.368)circle(0.1cm);
\draw(3.759,1.368)--(-2,-3.464);
\draw(3.759,1.368)--(3.064,-2.571);
\draw(3.759,1.368)--(3.064,2.571);
\draw(3.759,1.368)--(3.759,-1.368);
\draw(3.759,1.368)--(-3.759,-1.368);

\filldraw[fill=black](-3.759,1.368)circle(0.1cm);
\draw(-3.759,1.368)--(-2,3.464);
\draw(-3.759,1.368)--(2,3.464);
\draw(-3.759,1.368)--(-3.064,-2.571);
\draw(-3.759,1.368)--(-3.064,2.571);
\draw(-3.759,1.368)--(-3.759,-1.368);

\filldraw[fill=black](3.759,-1.368)circle(0.1cm);
\draw(3.759,-1.368)--(0.695,-3.939);
\draw(3.759,-1.368)--(2,-3.464);
\draw(3.759,-1.368)--(3.064,-2.571);
\draw(3.759,-1.368)--(3.064,2.571);

\filldraw[fill=black](-3.759,-1.368)circle(0.1cm);
\draw(-3.759,-1.368)--(-2,-3.464);
\draw(-3.759,-1.368)--(3.064,-2.571);
\draw(-3.759,-1.368)--(-3.064,-2.571);

\filldraw[fill=black](3.064,2.571)circle(0.1cm);
\draw(3.064,2.571)--(2,3.464);
\draw(3.064,2.571)--(0.695,3.939);
\draw(3.064,2.571)--(-0.695,3.939);

\filldraw[fill=black](-3.064,2.571)circle(0.1cm);
\draw(-3.064,2.571)--(-2,3.464);
\draw(-3.064,2.571)--(-0.695,3.939);
\draw(-3.064,2.571)--(2,3.464);

\filldraw[fill=black](3.064,-2.571)circle(0.1cm);
\draw(3.064,-2.571)--(2,-3.464);
\draw(3.064,-2.571)--(0.695,-3.939);
\draw(3.064,-2.571)--(-2,-3.464);

\filldraw[fill=black](-3.064,-2.571)circle(0.1cm);
\draw(-3.064,-2.571)--(-2,-3.464);
\draw(-3.064,-2.571)--(-0.695,-3.939);
\draw(-3.064,-2.571)--(0.695,-3.939);

\filldraw[fill=black](2,3.464)circle(0.1cm);
\draw(2,3.464)--(-2,3.464);
\draw(2,3.464)--(-0.695,3.939);
\draw(2,3.464)--(0.695,3.939);

\filldraw[fill=black](-2,3.464)circle(0.1cm);
\draw(-2,3.464)--(2,-3.464);
\draw(-2,3.464)--(-0.695,-3.939);
\draw(-2,3.464)--(0.695,3.939);

\filldraw[fill=black](2,-3.464)circle(0.1cm);
\draw(2,-3.464)--(0.695,-3.939);
\draw(2,-3.464)--(-0.695,-3.939);
\draw(2,-3.464)--(0.695,3.939);

\filldraw[fill=black](-2,-3.464)circle(0.1cm);
\draw(-2,-3.464)--(-0.695,-3.939);
\draw(-2,-3.464)--(0.695,-3.939);

\filldraw[fill=black](0.695,3.939)circle(0.1cm);
\draw(0.695,3.939)--(-0.695,3.939);
\draw(0.695,3.939)--(-0.695,-3.939);

\filldraw[fill=black](-0.695,3.939)circle(0.1cm);

\filldraw[fill=black](0.695,-3.939)circle(0.1cm);
\draw(0.695,-3.939)--(-0.695,-3.939);

\filldraw[fill=black](-0.695,-3.939)circle(0.1cm);

\filldraw[fill=black](8,1)circle(0.1cm);
\draw(4,0)--(8,1);
\draw(-4,0)--(8,1);
\draw(3.759,1.368)--(8,1);
\draw(-3.759,1.368)--(8,1);
\draw(3.759,-1.368)--(8,1);
\draw(-3.759,-1.368)--(8,1);
\draw(3.064,2.571)--(8,1);
\draw(-3.064,2.571)--(8,1);
\draw(3.064,-2.571)--(8,1);
\draw(-3.064,-2.571)--(8,1);
\draw(2,3.464)--(8,1);
\draw(-2,3.464)--(8,1);
\draw(2,-3.464)--(8,1);
\draw(-2,-3.464)--(8,1);
\draw(0.695,3.939)--(8,1);
\draw(-0.695,3.939)--(8,1);
\draw(0.695,-3.939)--(8,1);
\draw(-0.695,-3.939)--(8,1);

\filldraw[fill=black](17,0)circle(0.1cm); 
\draw(17,0)--(16.759,1.632);
\draw(17,0)--(16.759,-1.368);
\draw(17,0)--(16.064,-2.571);
\draw(17,0)--(15,-3.464);
\draw(17,0)--(13.695,-3.939);
\draw(17,0)--(9.936,2.571);

\filldraw[fill=black](9,0)circle(0.1cm); 
\draw(9,0)--(9.241,1.632);
\draw(9,0)--(11,3.464);
\draw(9,0)--(15,3.464);
\draw(9,0)--(16.064,2.571);
\draw(9,0)--(9.241,-1.368);
\draw(9,0)--(15,-3.464);

\filldraw[fill=black](16.759,1.632)circle(0.1cm);
\draw(16.759,1.632)--(16.064,2.571);
\draw(16.759,1.632)--(13.695,-3.939);
\draw(16.759,1.632)--(12.305,-3.939);
\draw(16.759,1.632)--(9.241,-1.368);
\draw(16.759,1.632)--(9.936,2.571);

\filldraw[fill=black](9.241,1.632)circle(0.1cm);
\draw(9.241,1.632)--(9.241,-1.368);
\draw(9.241,1.632)--(9.936,-2.571);
\draw(9.241,1.632)--(15,-3.464);
\draw(9.241,1.632)--(16.759,-1.368);
\draw(9.241,1.632)--(13.695,3.939);

\filldraw[fill=black](16.759,-1.368)circle(0.1cm);
\draw(16.759,-1.368)--(16.064,-2.571);
\draw(16.759,-1.368)--(15,-3.464);
\draw(16.759,-1.368)--(9.936,-2.571);
\draw(16.759,-1.368)--(13.695,3.939);

\filldraw[fill=black](9.241,-1.368)circle(0.1cm);
\draw(9.241,-1.368)--(16.064,2.571);
\draw(9.241,-1.368)--(15,-3.464);
\draw(9.241,-1.368)--(12.305,-3.939);

\filldraw[fill=black](16.064,2.571)circle(0.1cm);
\draw(16.064,2.571)--(15,3.464);
\draw(16.064,2.571)--(11,3.464);
\draw(16.064,2.571)--(12.305,-3.939);

\filldraw[fill=black](9.936,2.571)circle(0.1cm);
\draw(9.936,2.571)--(11,3.464);
\draw(9.936,2.571)--(12.305,3.939);
\draw(9.936,2.571)--(16.064,-2.571);
\draw(9.936,2.571)--(13.695,-3.939);

\filldraw[fill=black](16.064,-2.571)circle(0.1cm);
\draw(16.064,-2.571)--(15,-3.464);
\draw(16.064,-2.571)--(12.305,3.939);
\draw(16.064,-2.571)--(11,3.464);

\filldraw[fill=black](9.936,-2.571)circle(0.1cm);
\draw(9.936,-2.571)--(11,-3.464);
\draw(9.936,-2.571)--(12.305,-3.939);
\draw(9.936,-2.571)--(13.695,-3.939);
\draw(9.936,-2.571)--(13.695,3.939);

\filldraw[fill=black](15,3.464)circle(0.1cm);
\draw(15,3.464)--(13.695,3.939);
\draw(15,3.464)--(12.305,3.939);
\draw(15,3.464)--(11,3.464);
\draw(15,3.464)--(11,-3.464);

\filldraw[fill=black](11,3.464)circle(0.1cm);
\draw(11,3.464)--(12.305,3.939);

\filldraw[fill=black](15,-3.464)circle(0.1cm);

\filldraw[fill=black](11,-3.464)circle(0.1cm);
\draw(11,-3.464)--(12.305,-3.939);
\draw(11,-3.464)--(13.695,-3.939);
\draw(11,-3.464)--(12.305,3.939);
\draw(11,-3.464)--(13.695,3.939);

\filldraw[fill=black](13.695,3.939)circle(0.1cm);
\draw(13.695,3.939)--(12.305,3.939);

\filldraw[fill=black](12.305,3.939)circle(0.1cm);

\filldraw[fill=black](13.695,-3.939)circle(0.1cm);
\draw(13.695,-3.939)--(12.305,-3.939);

\filldraw[fill=black](12.305,-3.939)circle(0.1cm);
\filldraw[fill=black](21,1)circle(0.1cm);
\draw(17,0)--(21,1);
\draw(9,0)--(21,1);
\draw(16.759,1.632)--(21,1);
\draw(9.241,1.632)--(21,1);
\draw(16.759,-1.368)--(21,1);
\draw(9.241,-1.368)--(21,1);
\draw(16.064,2.571)--(21,1);
\draw(9.936,2.571)--(21,1);
\draw(16.064,-2.571)--(21,1);
\draw(9.936,-2.571)--(21,1);
\draw(15,3.464)--(21,1);
\draw(11,3.464)--(21,1);
\draw(15,-3.464)--(21,1);
\draw(11,-3.464)--(21,1);
\draw(13.695,3.939)--(21,1);
\draw(12.305,3.939)--(21,1);
\draw(13.695,-3.939)--(21,1);
\draw(12.305,-3.939)--(21,1);

\end{tikzpicture}
 \caption{non isomorphic non $\epsilon-$cospectral $\epsilon$equienergetic graphs on 19-vertices with $E_{\epsilon}=81.579$}
\label{fig:3}
\end{figure}
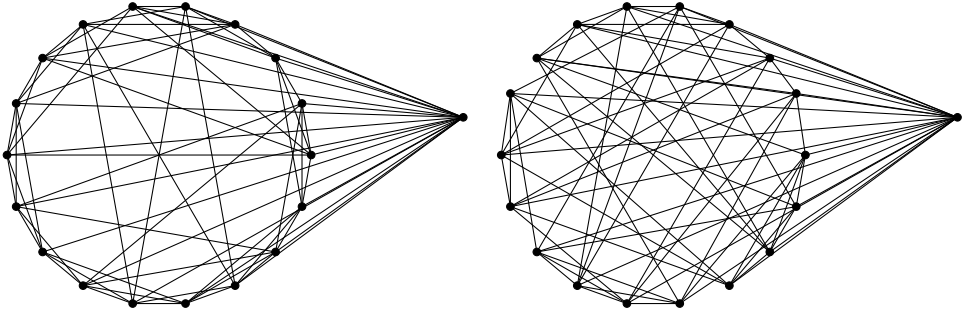

\noindent Next, we give a method for the construction of new pairs and triplets of non $\epsilon-$    cospectral $\epsilon-$equienergetic graphs.
\begin{theorem}\label{equienergetic}
    Let $G_{1}$ and $G_{2}$ be two  non cospectral cubic graph on $2t$ vertices, $t\geq2$, and $G$ be any $r-$ regular graph with $r\geq 2$. Let $H_{1}=L^{2}(G_{1})$ and $H_{2}=L^{2}(G_{2})$.
    Then \begin{enumerate}
        \item $G\veedot H_{1}$ and $G\veedot H_{2}$ are non $\epsilon-$cospectral $\epsilon-$equienergetic graphs.
        \item $G\veebar H_{1}$ and $G\veebar H_{2}$ are non $\epsilon-$cospectral $\epsilon-$equienergetic graphs.
        \item $G\veedot (H_{1}\cup H_{2})$ and $G\veedot (H_{1}\cup H_{1})$ and $G\veedot(H_{2}\cup H_{2})$ are non $\epsilon-$cospectral $\epsilon-$equienergetic graphs.
         \item $G\veebar (H_{1}\cup H_{2})$ and $G\veebar (H_{1}\cup H_{1})$ and $G\veebar(H_{2}\cup H_{2})$ are non $\epsilon-$cospectral $\epsilon-$equienergetic graphs.
    \end{enumerate}
\end{theorem}
\begin{proof}
The proof follows from Lemma \ref{2linegraph}, and  Theorems \ref{subdivisionvertex}, \ref{subdivisionedge}, \ref{subdivisionvertexunion}, \ref{subdivisionedgeunion}.
\end{proof}

\begin{corollary}
  For every $t\geq 3$, there exist 2 pairs of  $non$ $\epsilon-$cospectral $\epsilon-$equienergetic  graph on $6+6t$ vertices.  
\end{corollary}
\begin{proof}
   By  Theorem \ref{equienergetic} we have \begin{align*}
   E_{\epsilon}(K_{3}\veebar H_{1})=E_{\epsilon}(K_{3}\veebar H_{2})\\
   \text{ and } 
     E_{\epsilon}(K_{3}\veedot H_{1})=E_{\epsilon}(K_{3}\veedot H_{2}).
\end{align*}
\end{proof}
\begin{corollary}
    For every $t\geq 3$, there exist $2$ triplets of $non$ $\epsilon-$cospectral $\epsilon-$equienergetic  graphs on $6+12t$ vertices.
\end{corollary}
\begin{proof}
By  Theorem \ref{equienergetic} we have
    \begin{align*}
        E_{\epsilon}(K_{3}\veebar(H_{1}\cup H_{1}))=E_{\epsilon}(K_{3}\veebar(H_{2}\cup H_{2}))=E_{\epsilon}(K_{3}\veebar(H_{1}\cup H_{2})),\\
    \text{ and } 
        E_{\epsilon}(K_{3}\veedot (H_{1}\cup H_{1}))=E_{\epsilon}(K_{3}\veedot (H_{2}\cup H_{2}))=E_{\epsilon}(K_{3}\veedot (H_{1}\cup H_{2})).
    \end{align*}
\end{proof}

\subsection{Some $\epsilon-$integral graphs}
The problem of finding $A-$integral graphs is also an interesting one in spectral graph theory. Motivated by this, in this section, using the results obtained in previous sections, we give some methods to construct some new families of $\epsilon-$integral graphs.


\begin{proposition}
    Let $G_{j}$ be a $r_{j}$ ($r_{1}\geq 2$) regular, $(p_{j}, q_{j})$ graph. Then $G_{1}\veedot G_{2}$ is $\epsilon-$integral graph if and only if  $G_{2}$ is $A-$integral,  roots of the equation $3t^{2}+4t((\lambda_{i}+r_{1})-1)-3(\lambda_{i}+r_{1})=0$, $i=2,..., p_{1}$, and eigenvalues of the matrix 
    $\begin{pmatrix}
    0&3q_{1}-3r_{1}&0\\
    3p_{1}-6&4q_{1}-8r_{1}+4&2p_{2}\\
    0&2q_{1}&2(p_{2}-r_{2}-1
    \end{pmatrix}$ 
    are integers.
\end{proposition}
\begin{proposition}
    Let $G_{j}$ be a $r_{j}$ ($r_{1}\geq 2$) regular graph,  ($p_{i}, q_{i}$) graph. Then $G_{1}\veebar G_{2}$ is $\epsilon-$integral graph if and only if $G_{2}$ is $A-$integral and roots of the equation $3t^{2}-4(1+\lambda_{i})t-3(\lambda_{i}+r_{1})=0, i=2,\ldots, p_{1}$, and eigenvalues of the matrix 
    $\begin{pmatrix}
    4(p_{1}-1-r_{1})&3(q_{1}-r_{1})&2p_{2}\\
    3(p_{1}-2)&0&0\\
    2p_{1}&0&2(p_{2}-r_{2}-1)
    \end{pmatrix}$ 
    are integers.
\end{proposition}

Now we  present some $\epsilon-$integral graphs,
\begin{enumerate}
    \item The graph $K_{3}\veedot K_{n}$ is $\epsilon-$integral graphs if and only if $12n+9$ is a perfect square.\\
    For example $K_{3}\veedot K_{6}$, $K_{3}\veedot K_{18}$, $K_{3}\veedot K_{36}$, $K_{3}\veedot K_{60}$, $K_{3}\veedot K_{90}$, $K_{3}\veedot K_{126}$.
    \item The graph $K_{11}\veebar K_{n}$ is $\epsilon-$integral if and only if $44n+3645$ is a perfect square.\\
    For example  $K_{11}\veebar K_{45}$ is $\epsilon-$integral.
    \item Using Theorem \ref{joinunion} \begin{itemize}
        \item If $G$ is a $A-$integral graph then $G\vee (G\cup G)$ is $\epsilon-$integral.
        \item If $G_{i}$, $i=0,1,2$, $r-$ regular, $A-$integral graph then $G_{0}\vee (G_{1}\cup G_{2})$ is $\epsilon-$integral.
        \end{itemize}
    The graphs $K_{n}\vee (K_{m}\cup K_{l})$, $K_{m,m}\vee (K_{m+1}\cup K_{m+1})$, $K_{m+1}\vee (K_{m,m}\cup K_{m,m})$, $K_{m,m}\vee (K_{m+1}\cup K_{m,m})$ are some $\epsilon-$integral graphs.

\end{enumerate}

\section{Conclusion}
Using the graph operations subdivision vertex(edge) join and join of graphs, 
we have added new classes of connected graphs to the existing classes of graphs for which the irreducibility or reducibility of their eccentricity matrix is known. 
 In addition, we analyzed their eccentricity spectra. These results enable us to provide infinitely many  $\epsilon-$cospectral graph pairs. 
  Furthermore, infinitely many pairs and triplets of non $\epsilon-$cospectral $\epsilon-$equienergetic graphs are constructed. Moreover, some new families of  $\epsilon-$integral graphs are obtained.
  
\textbf{Word count: 3178}

\section*{Disclosure statement}
All authors declare that they have no conflicts of interest.

\bibliographystyle{plain}
 \bibliography{ref}
\end{document}